\documentclass[11pt]{amsart}

\usepackage{amsmath, amsthm, amssymb, a4wide}

\title[Polynomial Maps and Their Fixed-Point Multipliers: II. Improvement to Algorithm%
]{The Moduli Space of Polynomial Maps and Their Fixed-Point Multipliers: II. 
Improvement to the Algorithm and Monic Centered Polynomials}
\author{Toshi Sugiyama}
\thanks{Published online in Ergodic Theory Dynam. Systems (2023) pp.1--19}
\date{\today}

\address{Mathematics Studies, Gifu Pharmaceutical University, Mitahora-higashi 5-6-1, Gifu-city, Gifu 502-8585, Japan}
\email{sugiyama-to@gifu-pu.ac.jp}

\subjclass[2010]{Primary 37F10; Secondary 05A19, 14D20%
}
\keywords{complex dynamics, fixed-point multipliers, moduli space of polynomial maps, monic centered polynomials, partition of integers, combinatorial identities, inclusion-exclusion formula}

\thanks{This work was supported by JSPS KAKENHI Grant Number JP19K14557}

\theoremstyle{plain}
\newtheorem{mtheorem}{Theorem}

\newtheorem{mcorollary}[mtheorem]{Corollary}
\newtheorem{theorem}{Theorem}[section]
\newtheorem{lemma}[theorem]{Lemma}
\newtheorem{proposition}[theorem]{Proposition}

\theoremstyle{definition}
\newtheorem{remark}[theorem]{Remark}
\newtheorem{definition}[theorem]{Definition}
\newtheorem{example}[theorem]{Example}

\numberwithin{equation}{section}

\begin{document}

\begin{abstract}
 We consider the family $\mathrm{MC}_d$ of monic centered
 polynomials of one complex variable with degree $d \geq 2$, and 
 study the map $\widehat{\Phi}_d:\mathrm{MC}_d\to \widetilde{\Lambda}_d \subset \mathbb{C}^d / \mathfrak{S}_d$
 which maps each $f \in \mathrm{MC}_d$ to its unordered collection of fixed-point
 multipliers.
 We give an explicit formula for counting the number of elements
 of each fiber $\widehat{\Phi}_d^{-1}\left(\bar{\lambda}\right)$
 for every $\bar{\lambda} \in \widetilde{\Lambda}_d$
 except when the fiber $\widehat{\Phi}_d^{-1}\left(\bar{\lambda}\right)$
 contains polynomials having multiple fixed points.
 This formula is not a recursive one, 
 and is a drastic improvement of our previous result
 [T. Sugiyama,  The moduli space of polynomial maps and their fixed-point multipliers.  
  \textit{Adv. Math.} 322  (2017),  132--185]
 which gave a rather long algorithm with some induction processes.
\end{abstract}

\maketitle

\section{Introduction}\label{secintroduction}

This paper is a continuation of the author's previous work~\cite{sugi1}.

We first remind our setting from~\cite{sugi1}.
Let $\mathrm{MP}_d$ be the family of affine conjugacy classes of polynomial maps of one complex variable with degree $d \geq 2$,
and $\mathbb{C}^d / \mathfrak{S}_d$ the set of unordered collections of $d$ complex numbers, where ${\mathfrak S}_{d}$ denotes the $d$-th symmetric group.
We denote by $\Phi_d$ the map 
\[
\Phi_d : \mathrm{MP}_d \to \widetilde{\Lambda}_d \subset \mathbb{C}^d / \mathfrak{S}_d
\]
which maps  each $f \in \mathrm{MP}_d$ to its unordered collection of fixed-point multipliers.
Here, fixed-point multipliers of $f \in \mathrm{MP}_d$ always satisfy a certain relation by the fixed point theorem for polynomial maps (see Section~12 in~\cite{mi_book}), which implies that
the image of $\Phi_d$ is contained in a certain hyperplane $\widetilde{\Lambda}_d$ in $\mathbb{C}^d / \mathfrak{S}_d$.

As mentioned in~\cite{sugi1}, it is well known that the map $\Phi_d : \mathrm{MP}_d \to \widetilde{\Lambda}_d$ is bijective for $d=2$ and also for $d=3$ (see~\cite{mi_cub}).
For $d \geq 4$, Fujimura and Nishizawa
have done some preliminary works
in finding $\#\left(\Phi_d^{-1}(\bar{\lambda})\right)$ for $\bar{\lambda} \in \widetilde{\Lambda}_d$ in their series of papers such as~\cite{NishizawaFujimura},~\cite{Fujimura2} and~\cite{fu}.
Hereafter, $\#(X)$, or simply $\# X$, denotes the cardinality of a set $X$.
Fujimura and Taniguchi~\cite{tani} also constructed a compactification of $\mathrm{MP}_d$,
which gave us a strong geometric insight on the fiber structure of $\Phi_d$.
Other compactifications of $\mathrm{MP}_d$ were also constructed independently
by Silverman~\cite{sil} and by DeMarco and McMullen~\cite{demarco}.
For rational maps and their periodic-point multipliers, McMullen~\cite{Mc} gave a general important result. 
In a special case of~\cite{Mc}, 
there is a famous result by Milnor~\cite{mi_qua} for rational maps of degree two and their fixed-point multipliers.
There is also a result by Hutz and Tepper~\cite{HutzTepper} for rational maps of degree three
and their periodic-point multipliers of period less than or equal to two.
There are some other results~\cite{Gorbovickis1},~\cite{Gorbovickis2} concerning polynomial or rational maps and their periodic-point multipliers. 
(See~\cite{sugi1} for more details.)

Following the results above, in~\cite{sugi1}, we succeeded in giving, for every $\bar{\lambda} = \left\{ \lambda_1,\dots, \lambda_d \right\} \in \widetilde{\Lambda}_d$,
an algorithm for counting the number of elements of $\Phi_d^{-1}(\bar{\lambda})$
except when $\lambda_i = 1$ for some $i$.
However, 
the algorithm was rather long and complicated.
In this paper, we make a {\it drastic improvement} to its algorithm;  
we no longer need induction processes to find $\#\left(\Phi_d^{-1}(\bar{\lambda})\right)$
if we consider $\Phi_d^{-1}(\bar{\lambda})$ counted with multiplicity (see Theorem~\ref{mainthm1}).
Moreover, if we consider the family $\mathrm{MC}_d$ of monic centered polynomials of degree $d$ and the map $\widehat{\Phi}_d : \mathrm{MC}_d \to \widetilde{\Lambda}_d$, instead of $\mathrm{MP}_d$ and $\Phi_d : \mathrm{MP}_d \to \widetilde{\Lambda}_d$,
we can always give an explicit expression of $\#\left(\widehat{\Phi}_d^{-1}(\bar{\lambda})\right)$ even when its multiplicity is ignored (see Theorem~\ref{mainthm2} and Corollary~\ref{cor3}).
Here, $\widehat{\Phi}_d : \mathrm{MC}_d \to \widetilde{\Lambda}_d$ is defined to be the composite mapping 
of the natural projection $\mathrm{MC}_d \to \mathrm{MP}_d$ and $\Phi_d$.
Interestingly, the formula for finding $\#\left(\Phi_d^{-1}(\bar{\lambda})\right)$ in Theorem~\ref{mainthm1} has the form of the inclusion-exclusion formula.

There are five sections and one appendix in this paper. 
In Sections~\ref{secmainresult1} and~\ref{secmainresult2},
we shall review the results in~\cite{sugi1} more precisely 
and state Theorems~\ref{mainthm1},~\ref{mainthm2} and Corollary~\ref{cor3},
which are the main results in this paper.
Section~\ref{proof1} is devoted to the proof of Theorem~\ref{mainthm1}, and 
Section~\ref{proof2} is devoted to the proof of Theorem~\ref{mainthm2}.
The main part in this paper is the proof of Theorem~\ref{mainthm1} in Section~\ref{proof1}, 
which consists of a good deal of combinatorial argument.
Compared with the proof of Theorem~\ref{mainthm1}, the proof of Theorem~\ref{mainthm2} in Section~\ref{proof2} is relatively easy under the assumption of~\cite{sugi1}.
However, by combining Theorems~\ref{mainthm1} and~\ref{mainthm2}, we directly have Corollary~\ref{cor3}, 
which is, in some sense, a monumental achievement of our study.
In Appendix, we 
explain
how to
find out the formula~(\ref{eq2.5}) in Theorem~\ref{mainthm1}.

\proof[\bf Acknowledgements]
The author would like to express his thanks to Professor Hiroki Sumi for 
valuable advices on this paper.

\section{Main result 1}\label{secmainresult1}

In this section, we always consider $\Phi_d^{-1}(\bar{\lambda})$ {\it counted with multiplicity}, and
deal with improvements to the algorithm for finding $\#\left(\Phi_d^{-1}(\bar{\lambda})\right)$.
We first fix our notation.

For $d \ge 2$, we put
\[
 \mathrm{Poly}_d := \left\{f \in \mathbb{C}[z] \bigm| \deg f = d  \right\}
 \quad \textrm{and} \quad
 \mathrm{Aut}(\mathbb{C})
 := \left\{\gamma (z) = az+b \bigm| a,b \in \mathbb{C},\ a \ne 0 \right\}.
\]
Since $\gamma \in \mathrm{Aut}(\mathbb{C})$
naturally acts on $f \in \mathrm{Poly}_d$
by $\gamma \cdot f := \gamma \circ f \circ \gamma^{-1}$,
we can define its quotient $\mathrm{MP}_d := \mathrm{Poly}_d / \mathrm{Aut}(\mathbb{C})$,
which we usually call the moduli space of polynomial maps
of degree $d$.
We put $\mathrm{Fix}(f) := \{ z \in \mathbb{C} \bigm| f(z)=z\}$ for $f \in \mathrm{Poly}_d$, where $\mathrm{Fix}(f)$ is considered counted with multiplicity.
Hence, we always have $\# \left(\mathrm{Fix}(f)\right) = d$.
Since the unordered collection of fixed-point multipliers $\left(f'(\zeta)\right)_{\zeta \in \mathrm{Fix}(f)}$ of $f \in \mathrm{Poly}_d$
is invariant under the action of $\mathrm{Aut}(\mathbb{C})$,
we can naturally define the map
$\Phi_d:\mathrm{MP}_d \to \mathbb{C}^{d}/\mathfrak{S}_{d}$
by $\Phi_d(f):=\left(f'(\zeta)\right)_{\zeta \in \mathrm{Fix}(f)}$.
Here, $\mathfrak{S}_{d}$ denotes the $d$-th symmetric group 
which acts on $\mathbb{C}^{d}$ by the permutation of coordinates.
Note that a fixed point $\zeta \in \mathrm{Fix}(f)$
is multiple if and only if $f'(\zeta)=1$.

By the fixed point theorem for polynomial maps, we always have $\sum_{\zeta \in \mathrm{Fix}(f)} \frac{1}{1-f'(\zeta)} = 0$
for $f \in \mathrm{Poly}_d$ if $f$ has no multiple fixed point. 
(See Section~12 in~\cite{mi_book} or Proposition~1.1 in~\cite{sugi1} for more datails.) 
Hence, putting $\Lambda_d :=\left\{(\lambda_1,\ldots ,\lambda_d)\in\mathbb{C}^d \ \left| \
  \sum_{i=1}^d \prod_{j\ne i} \left( 1- \lambda_j \right) = 0
 \right. \right\}$ and 
$\widetilde{\Lambda}_d := \Lambda_d / \mathfrak{S}_d$,
we have the inclusion relation 
$\Phi_d\left(\mathrm{MP}_d\right) \subseteq \widetilde{\Lambda}_d \subseteq \mathbb{C}^{d}/\mathfrak{S}_{d}$.
We therefore have the map
\[
 \Phi_d : \mathrm{MP}_d \to \widetilde{\Lambda}_d
\]
by $f \mapsto \left(f'(\zeta) \right)_{\zeta \in \mathrm{Fix}(f)}$, 
which is the main object of our study.

In this paper, we again restrict our attention to the map $\Phi_d$ on the domain
where polynomial maps have no multiple fixed points, that is,
on the domains
\[
  V_d := \left\{ (\lambda_1,\ldots ,\lambda_d) \in \Lambda_d \bigm|
  \lambda_i \ne 1 \ \textrm{for every} \ 1\le i\le d \right\} \ 
  \textrm{and} \ \widetilde{V}_d := V_d/ \mathfrak{S}_d,
\]
which are Zariski open subsets
of $\Lambda_d$ and $\widetilde{\Lambda}_d$, respectively.
Here, note that we also have
\[
  V_d = \left\{ (\lambda_1,\ldots ,\lambda_d) \in \mathbb{C}^d \Biggm|
  \lambda_i \ne 1 \ \textrm{for every} \ 1\le i\le d,\ \ \sum_{i=1}^d \frac{1}{1-\lambda_i} = 0 \right\}.
\]
Throughout this paper, we always denote by $\bar{\lambda}$
the equivalence class of $\lambda \in \Lambda_d$ in $\widetilde{\Lambda}_d$, that is, 
$\bar{\lambda}=\textit{pr}(\lambda)$, 
where $\textit{pr}: \Lambda_d \to \widetilde{\Lambda}_d$ denotes the canonical projection.
Hence, for $\lambda = (\lambda_1,\ldots ,\lambda_d) \in \Lambda_d$, 
we sometimes express $\bar{\lambda} = \left\{ \lambda_1,\dots, \lambda_d \right\} \in \widetilde{\Lambda}_d$.
We never denote by $\bar{\lambda}$ the complex conjugate of $\lambda$ in this paper.

The objects defined in the following definition play a central roll in~\cite{sugi1} and also in this paper.

\begin{definition}\label{def2.1}
 For $\lambda =  (\lambda_1,\ldots ,\lambda_d) \in V_d$, we put
 \[
   \mathfrak{I}(\lambda) := \left\{ \left\{I_1,\ldots,I_l\right\}\ \left| \ 
	 \begin{matrix}
	   l \ge 2,\ \ I_1 \amalg \cdots \amalg I_l = \{1,\ldots,d \},\\
	   I_u \ne \emptyset \textrm{ for every } 1\le u\le l,\\
	  \sum_{i \in I_u} 1/(1-\lambda_i)=0
	    \textrm{ for every } 1\le u\le l
	 \end{matrix}
	 \right.\right\},
 \]
 where $I_1 \amalg \cdots \amalg I_l$ denotes the disjoint union of $I_1,\ldots,I_l$.
 By definition, each element of $\mathfrak{I}(\lambda)$ is considered to be a partition of $\{1,\ldots,d \}$.
 The partial order $\prec$ in $\mathfrak{I}(\lambda)$ is defined
 by the refinement of partitions, namely, 
 for $\mathbb{I}, \mathbb{I}'\in\mathfrak{I}(\lambda)$,
 the relation $\mathbb{I} \prec \mathbb{I}'$ holds if and only if
 $\mathbb{I}'$ is a refinement of $\mathbb{I}$ as partitions of $\{1,\ldots,d \}$.

 For $\lambda \in V_d$ and for $I \in \mathbb{I} \in \mathfrak{I}(\lambda)$, we put $\lambda_I := (\lambda_i)_{i \in I}$.
\end{definition}

In the above definition, note that 
the condition $I \in \mathbb{I} \in \mathfrak{I}(\lambda)$ for $I$ is equivalent to the conditions 
$\emptyset \subsetneq I \subsetneq \{1,\ldots,d \}$
and $\sum_{i \in I}  1/(1-\lambda_i)=0$.
Hence, we always have $\lambda_I \in V_{\# I}$ for $\lambda \in V_d$ and $I \in \mathbb{I} \in \mathfrak{I}(\lambda)$ by definition.
Also note that $\# I \geq 2$ holds for every $I \in \mathbb{I} \in \mathfrak{I}(\lambda)$.

The following object is also very important in this paper.

\begin{definition}\label{def2.2}
 For $\lambda \in V_d$ we put
 \[
   \mathfrak{I}'(\lambda) := \mathfrak{I}(\lambda)  \cup \left\{\left\{\left\{ 1, \dots, d\right\}\right\}\right\}.
 \]
 The partial order $\prec$ in $\mathfrak{I}(\lambda)$ is naturally extended to the partial order $\prec$ in $\mathfrak{I}'(\lambda)$.
\end{definition}

By definition, $\mathfrak{I}'(\lambda)$ is obtained from $\mathfrak{I}(\lambda)$ 
by adding exactly one element $\mathbb{I}_0 := \left\{\left\{ 1, \dots, d\right\}\right\}$.
Here, $\mathbb{I}_0$ is the unique minimum element of $\mathfrak{I}'(\lambda)$ with respect to the partial order $\prec$. 
Moreover, $\mathbb{I}_0$ is considered to be a partition of $\{ 1, \dots, d \}$ which, in practice, does not partition $\{ 1, \dots, d \}$.
We also have the equality
\[
     \mathfrak{I}'(\lambda) = \left\{ \left\{I_1,\ldots,I_l\right\}\ \left| \ 
	 \begin{matrix}
	   l \ge 1,\ \ I_1 \amalg \cdots \amalg I_l = \{1,\ldots,d \},\\
	   I_u \ne \emptyset \textrm{ for every } 1\le u\le l,\\
	  \sum_{i \in I_u} 1/(1-\lambda_i)=0
	    \textrm{ for every } 1\le u\le l
	 \end{matrix}
	 \right.\right\}.
\]

We already have the following theorem
by Main Theorem~III and Remark~1.8 in~\cite{sugi1}
and by Theorem~B and Proposition~C in Section~6 in~\cite{sugi1}.

\begin{theorem}\label{thm2.3}
 We can define the non-negative integer $e_{\mathbb{I}}(\lambda)$ for each $d \geq 4$, $\lambda \in V_d$ and $\mathbb{I} \in \mathfrak{I}(\lambda)$, 
 and can also define the non-negative integer $s_d(\lambda)$ for each $d \geq 2$ and $\lambda \in V_d$ inductively by the equalities
 \begin{equation}\label{eq2.3}
    s_d(\lambda) = (d-2)!
	  - \sum_{\mathbb{I} \in \mathfrak{I}(\lambda)} \left(
          e_{\mathbb{I}}(\lambda)\cdot
          \prod_{k=d-\# \mathbb{I} +1}^{d-2}k \right)
 \end{equation}
 for $d \geq 2$ and $\lambda \in V_d$, and
 \begin{equation}\label{eq2.2}
    e_{\mathbb{I}}(\lambda)
   = \prod_{I \in \mathbb{I}} \bigl( 
   \left( \# I -1 \right) \cdot s_{\# I}\left(\lambda_{I}\right)
    \bigr)
 \end{equation}
 for $d \geq 4$, $\lambda \in V_d$ and $\mathbb{I} \in \mathfrak{I}(\lambda)$.
 Here, in the case $\# \mathbb{I} = 2$, we put 
 $\prod_{k=d-\# \mathbb{I} +1}^{d-2}k = \prod_{k=d-1}^{d-2}k = 1$.

 If we consider $\Phi_d^{-1}(\bar{\lambda})$ {\rm `counted with multiplicity'} for $d \geq 2$ and $\lambda \in V_d$, then we have
 \[
    \#\left(\Phi_d^{-1}(\bar{\lambda})\right) = s_d(\lambda).
 \]
\end{theorem}

\begin{remark}
 For $d=2$ or  $3$, we always have $\mathfrak{I}(\lambda) = \emptyset$ for every $\lambda \in V_d$ by definition.
 Hence, by equation~(\ref{eq2.3}), we have $s_2(\lambda) = (2-2)!=1$ for every $\lambda \in V_2$, and 
 $s_3(\lambda) = (3-2)!=1$ for every $\lambda \in V_3$.
 For $d \geq 4$, every $e_{\mathbb{I}}(\lambda)$ and $s_d(\lambda)$ are determined uniquely and can actually be found by equations~(\ref{eq2.3}) and~(\ref{eq2.2}) by induction on $d$, since $2 \leq \# I < d$ holds for $I \in \mathbb{I} \in \mathfrak{I}(\lambda)$ with $\lambda \in V_d$.
\end{remark}

In the rest of this paper, 
we always assume that $e_{\mathbb{I}}(\lambda)$ and $s_d(\lambda)$ are the non-negative integers defined in Theorem~\ref{thm2.3}.

We already made a minor improvement to the above algorithm
by Main Theorem~III in~\cite{sugi1} and by Proposition~D in Section~6 in~\cite{sugi1},
as in the following.

\begin{theorem}\label{thm2.5}
 The non-negative integer $e_{\mathbb{I}}(\lambda)$ for $\lambda \in V_d$ and $\mathbb{I} \in \mathfrak{I}(\lambda)$ defined in Theorem~\ref{thm2.3}
 also satisfies the equality
 \begin{equation}\label{eq2.4}
	 e_{\mathbb{I}}(\lambda) =
	 \left(\prod_{I \in \mathbb{I}} \bigl( \# I  -1 \bigr)! \right)
	  - \sum_\textrm{\scriptsize $\begin{matrix}
				       \mathbb{I}' \in \mathfrak{I}(\lambda) \\
				       \mathbb{I}' \succ \mathbb{I}, \;
				       \mathbb{I}' \ne \mathbb{I}
				     \end{matrix}$}
	 \left(
	  e_{\mathbb{I}'}(\lambda) \cdot \prod_{I \in \mathbb{I}}
	  \left( \prod_{k=\# I - \chi_I(\mathbb{I}')+1 }^{\# I -1}k \right)
	 \right),
 \end{equation}
	where we put
	$\chi_I(\mathbb{I}'):=\#\left(\left\{ I'\in\mathbb{I}' \bigm|
	I' \subseteq I \right\}\right)$
	for $\mathbb{I}' \succ \mathbb{I}$ and $I \in \mathbb{I}$.
	Here, in the case $\chi_I(\mathbb{I}') = 1$, we put 
	$\prod_{k=\# I - \chi_I(\mathbb{I}')+1 }^{\# I -1}k = \prod_{k=\# I}^{\# I -1}k = 1$.
\end{theorem}

\begin{remark}
 By definition, we always have $\sum_{I \in \mathbb{I}} \chi_I(\mathbb{I}') = \# \mathbb{I}'$ for $\mathbb{I}' \succ \mathbb{I}$.
\end{remark}

\begin{remark}
We can also find $s_d(\lambda)$ only
by using equations~(\ref{eq2.3}) and~(\ref{eq2.4}).
The algorithm using equations~(\ref{eq2.3}) and~(\ref{eq2.4}) is a little simpler than the algorithm in
 Theorem~\ref{thm2.3}.
\end{remark}

\begin{remark}
 We present a rough outline of the proof of Theorem~\ref{thm2.5} in this remark,
 since the proof can be an easy exercise for the proof of Theorem~\ref{mainthm1} in this paper.
 (See ``Proof of Proposition D" on p.175-177 in~\cite{sugi1} for details.)
 In the case where $d = \# I$ and $\lambda = \lambda_I$, equation~(\ref{eq2.3}) is equivalent to the following:
 \begin{equation}\label{eq2.7}
  \bigl( \# I - 1 \bigr)! = \bigl( \# I - 1 \bigr) s_{\# I}(\lambda_I) + \sum_{\mathbb{I} \in \mathfrak{I}(\lambda_I)} 
  \left( e_{\mathbb{I}}(\lambda_I)\cdot \prod_{k=\#I -\# \mathbb{I}+1}^{\#I - 1}k \right).
 \end{equation}
 Plugging equation~(\ref{eq2.7}) into $\prod_{I \in \mathbb{I}} \bigl( \# I  -1 \bigr)!$
 and using equation~(\ref{eq2.2}) carefully, we have equation~(\ref{eq2.4}).
\end{remark}

In this paper,
we make 
a drastic improvement to the above algorithm as in the following.

\begin{mtheorem}\label{mainthm1}
 The non-negative integer $s_d(\lambda)$ for $d \geq 2$ and $\lambda \in V_d$ defined in Theorem~\ref{thm2.3} is expressed in the form
 \begin{equation}\label{eq2.5}
    (d-1)s_d(\lambda) = \sum_{\mathbb{I}\in \mathfrak{I}'(\lambda)} 
				\left( \left\{ -(d-1) \right\}^{\#\mathbb{I} - 1} 
				\cdot \prod_{I \in \mathbb{I}}\left( \#I - 1 \right)! \right).
 \end{equation}
 Hence, if we consider $\Phi_d^{-1}(\bar{\lambda})$ {\rm `counted with multiplicity'} for $d \geq 2$ and $\lambda \in V_d$, then we have
 \begin{equation}\label{eq2.6}
    \#\left(\Phi_d^{-1}(\bar{\lambda})\right) = -\sum_{\mathbb{I}\in \mathfrak{I}'(\lambda)} 
				\left( \left\{ -(d-1) \right\}^{\#\mathbb{I} - 2} 
				\cdot \prod_{I \in \mathbb{I}}\left( \#I - 1 \right)! \right).
 \end{equation}
\end{mtheorem}

Theorem~\ref{mainthm1} is proved in Section~\ref{proof1}, whereas
an intuitional consideration of Theorem~\ref{mainthm1} is added in Appendix.

\begin{remark}
 By Theorem~\ref{mainthm1}, we no longer need induction processes to find $\#\left(\Phi_d^{-1}(\bar{\lambda})\right)$ if we consider $\Phi_d^{-1}(\bar{\lambda})$ counted with multiplicity.
We only need to find $\mathfrak{I}'(\lambda)$ and to compute straightforward the right-hand side of equation~(\ref{eq2.6}).

However, there are some minor defects in the form of equation~(\ref{eq2.6}) comparing with equation~(\ref{eq2.3}).
By equation~(\ref{eq2.3}), we can easily see the inequality $s_d(\lambda) \leq (d-2)!$; however, it cannot be easily seen by equation~(\ref{eq2.6}).
The sum of the absolute value
\[ 
	\sum_{\mathbb{I}\in \mathfrak{I}'(\lambda)} 
				\left( (d-1)^{\#\mathbb{I} - 2} 
				\cdot \prod_{I \in \mathbb{I}}\left( \#I - 1 \right)! \right)
\]
in the right-hand side of equation~(\ref{eq2.6}) can be much greater than $(d-2)!$.
\end{remark}

\begin{remark}
 Each term 
 in the right-hand side of equation~(\ref{eq2.5}) 
 \[
    \left\{ -(d-1) \right\}^{\#\mathbb{I} - 1} \cdot \prod_{I \in \mathbb{I}}\left( \#I - 1 \right)!
 \]
 is positive or negative, according to whether $\#\mathbb{I}$ is odd or even.
 Moreover, if $\mathbb{I} \in \mathfrak{I}'(\lambda)$ and $\mathbb{I}' \prec \mathbb{I}$,
 then we automatically have $\mathbb{I}' \in \mathfrak{I}'(\lambda)$.
 Hence, equation~(\ref{eq2.5}) is considered to be a kind of inclusion-exclusion formula.
\end{remark}

\begin{remark}
 Theorem~\ref{mainthm1} is derived from Theorem~\ref{thm2.3} with no extra information.
 Hence, the proof of Theorem~\ref{mainthm1} is self-contained and requires no prerequisites
 under the assumption of  Theorem~\ref{thm2.3},
 whereas its proof is highly non-trivial.
 The proof consists of a good deal of combinatorial argument.
\end{remark}

\section{Main result 2}\label{secmainresult2}

In this section, 
we proceed to the next step, in which we 
discuss the possibility of improving the algorithm for counting the number of {\it discrete} elements of $\Phi_d^{-1}(\bar{\lambda})$.
Therefore, in 
this section, $\Phi_d^{-1}(\bar{\lambda})$ is {\it not} considered counted with multiplicity; 
$\Phi_d^{-1}(\bar{\lambda})$ is considered to be a set.
In this setting,
we have already obtained an algorithm for counting the number of discrete elements of $\Phi_d^{-1}(\bar{\lambda})$ 
by using $\left\{s_{d'}(\lambda') \mid 2 \leq d' \leq d,\ \lambda' \in V_{d'} \right\}$
in the third and fourth steps in Main Theorem~III in~\cite{sugi1}.
To review the result more precisely and to discuss further properties, we first fix our notation.

The following objects are important in this section.

\begin{definition}\label{def3.1}
 For $\lambda = (\lambda_1, \dots, \lambda_d) \in V_d$, we put
 \[
    \mathfrak{K}(\lambda) := 
	\left\{
		K \ \left| \ 
	 \begin{matrix}
	   \emptyset \subsetneq K \subseteq \{1,\ldots,d \},\\
	   i,j \in K \Rightarrow \lambda_i = \lambda_j,\\
	   i\in K,\ j\in\{1,\ldots,d\}\setminus K\Longrightarrow\lambda_i\ne\lambda_j
	 \end{matrix}
	 \right.
	\right\}.
 \]
\end{definition}

Note that if we put $\mathfrak{K}(\lambda) =: \left\{ K_1,\dots, K_q \right\}$, then $K_1,\dots, K_q$ are mutually disjoint, and
the equality $K_1 \amalg \cdots \amalg K_q = \{1,\ldots,d\}$ holds by definition;
hence $\mathfrak{K}(\lambda)$ is a partition of $\{1,\ldots,d\}$.

\begin{definition}
 We denote the family of monic centered polynomials of degree $d$ by 
 \[
    \mathrm{MC}_d := \left\{\left. f(z) = z^d + \sum_{k=0}^{d-2}a_kz^k \ \right| \ a_k \in \mathbb{C} \ \ \text{for} \ \ 0 \leq k \leq d-2 \right\}.
 \]
 Moreover, we denote the composite mapping of $\mathrm{MC}_d \subset \mathrm{Poly}_d \twoheadrightarrow 
 \mathrm{Poly}_d / \mathrm{Aut}(\mathbb{C}) = \mathrm{MP}_d$ by $p : \mathrm{MC}_d \to \mathrm{MP}_d$, and also
  denote the composite mapping of $p: \mathrm{MC}_d \to \mathrm{MP}_d$ and 
 $\Phi_d : \mathrm{MP}_d \to \widetilde{\Lambda}_d$ by $\widehat{\Phi}_d : \mathrm{MC}_d \to \widetilde{\Lambda}_d$, that is, 
 $\widehat{\Phi}_d := \Phi_d \circ p$.
\end{definition}

In the above definition, the map $p$ is surjective since 
every affine conjugacy class of polynomial maps contains monic centered polynomials.
Moreover, two monic centered polynomials $f, g \in \mathrm{MC}_d$ are affinely conjugate 
if and only if there exists a $(d-1)$-th radical root $a$ of $1$ such that the equality $g(z) = a f(a^{-1}z)$ holds.
Hence, the group $\{ a \in \mathbb{C} \mid a^{d-1} = 1 \} \cong \mathbb{Z} /(d-1)\mathbb{Z}$
naturally acts on $\mathrm{MC}_d$, and the induced mapping 
$\overline{p} : \mathrm{MC}_d / \left( \mathbb{Z} /(d-1)\mathbb{Z} \right) \to \mathrm{MP}_d$ is an isomorphism.
Since $\mathrm{MC}_d \cong \mathbb{C}^{d-1}$, 
we also have $\mathrm{MP}_d \cong \mathbb{C}^{d-1} / \left( \mathbb{Z} /(d-1)\mathbb{Z} \right)$.
Here, the action of $\mathbb{Z} /(d-1)\mathbb{Z}$ on $\mathrm{MC}_d$ is {\it not} free for $d \geq 3$, and
$\mathrm{MP}_d$ has the set of singular points $\mathrm{Sing}(\mathrm{MP}_d)$ for $d \geq 4$.
Hence, in some sense, the map $p : \mathrm{MC}_d \to \mathrm{MP}_d$ can be considered 
to be a `desingularization' of $\mathrm{MP}_d$ for $d \geq 4$.

We already have the following theorem by Remark~1.9 in~\cite{sugi1}.

\begin{theorem}\label{thm3.3}
 For $d \geq 2$ and $\lambda \in V_d$, we put $\mathfrak{K}(\lambda) =: \left\{ K_1,\dots, K_q \right\}$ and
 denote by $g_w$ the greatest common divisor of 
 $\#K_1,\ldots,\#K_{(w-1)},(\#K_w)-1,\#K_{(w+1)},\ldots,\#K_q$
 for each $1\le w\le q$.
 If $g_w=1$ holds for every $1 \leq w \leq q$, then we have 
 \begin{equation}\label{eq3.1}
     \#\left(\Phi_d^{-1}(\bar{\lambda})\right) = \frac{s_d(\lambda)}{\left(\#K_1\right)!\dots\left(\#K_q\right)!}
     = \frac{s_d(\lambda)}{\prod_{K \in \mathfrak{K}(\lambda)}\left(\#K\right)!},
 \end{equation}
 where $s_d(\lambda)$ is the non-negative integer defined in Theorem~\ref{thm2.3} and rewritten in Theorem~\ref{mainthm1}.
 Here, note that $\Phi_d^{-1}(\bar{\lambda})$ is not considered counted with multiplicity, and
 hence $\#\left(\Phi_d^{-1}(\bar{\lambda})\right)$ denote the number of discrete elements of $\Phi_d^{-1}(\bar{\lambda})$.
\end{theorem}

In the case where $g_w \geq 2$ for some $w$, we also have an algorithm 
for finding $\#\left(\Phi_d^{-1}(\bar{\lambda})\right)$ in the third and fourth steps in Main Theorem~III in~\cite{sugi1}.
However, it contains induction processes and
 is much more complicated than equation~(\ref{eq3.1});
and hence we omit to describe it again in this paper.

As we already mentioned in Remark~1.9 in~\cite{sugi1},
we find that for $d \geq 4$ and for $\lambda \in V_d$, the inequality $g_w \geq 2$ holds for some $w$ 
only if $\bar{\lambda} \in \Phi_d(\mathrm{Sing}(\mathrm{MP}_d))$.
Since $\mathrm{MC}_d$ is a `desingularization' of $\mathrm{MP}_d$, it is natural to expect that 
the map
$\widehat{\Phi}_d = \Phi_d\circ p : \mathrm{MC}_d \to \widetilde{\Lambda}_d$
is simpler than the map $\Phi_d : \mathrm{MP}_d \to \widetilde{\Lambda}_d$ itself.
In the following, we consider $\mathrm{MC}_d$ instead of $\mathrm{MP}_d$, and 
also consider $\widehat{\Phi}_d : \mathrm{MC}_d \to \widetilde{\Lambda}_d$ instead of 
$\Phi_d : \mathrm{MP}_d \to \widetilde{\Lambda}_d$.

We now state the second main theorem in this paper.

\begin{mtheorem}\label{mainthm2}
 For $d \geq2$, $\lambda \in V_d$ and $\widehat{\Phi}_d : \mathrm{MC}_d \to \widetilde{\Lambda}_d$, we have
 \begin{equation}\label{eq3.2}
    \#\left( \widehat{\Phi}_d^{-1}(\bar{\lambda}) \right) 
    = \frac{(d-1)s_d(\lambda)}{\prod_{K \in \mathfrak{K}(\lambda)}\left(\#K\right)!},
 \end{equation}
 where $s_d(\lambda)$ is the non-negative integer defined in Theorem~\ref{thm2.3} and rewritten in Theorem~\ref{mainthm1}.
 Here, note that $\widehat{\Phi}_d^{-1}(\bar{\lambda})$ is not considered counted with multiplicity, and
 hence $\#\left(\widehat{\Phi}_d^{-1}(\bar{\lambda})\right)$ denotes the number of discrete elements of $\widehat{\Phi}_d^{-1}(\bar{\lambda})$.
\end{mtheorem}

Theorem~\ref{mainthm2} is proved in Section~\ref{proof2}.

\begin{remark}
 Theorem~\ref{mainthm2} holds for {\it every} $\lambda \in V_d$ with no 
 exception, and
 has no induction process.
 Hence, we can say that the fiber structure of the map $\widehat{\Phi}_d : \mathrm{MC}_d \to \widetilde{\Lambda}_d$ 
 is simpler than the fiber structure of the map $\Phi_d: \mathrm{MP}_d \to \widetilde{\Lambda}_d$, 
 or moreover we can also say that the complexity of the map $\Phi_d: \mathrm{MP}_d \to \widetilde{\Lambda}_d$ is composed of the two complexities: 
 one of them is the complexity of the map  $\widehat{\Phi}_d : \mathrm{MC}_d \to \widetilde{\Lambda}_d$, 
 and the other is the complexity of the map $p: \mathrm{MC}_d \to \mathrm{MP}_d$.
 Therefore, in some sense, 
 consideration of the map $\widehat{\Phi}_d$ 
 is more essential than that of the map $\Phi_d$ in the study of fixed-point multipliers for polynomial maps.
\end{remark}

\begin{remark}
 Theorem~\ref{mainthm2} is proved by a closer look at Propositions~4.3 and~9.1 in~\cite{sugi1}.
\end{remark}

Combining Theorems~\ref{mainthm1} and~\ref{mainthm2}, we have the following.

\begin{mcorollary}\label{cor3}
 For $d \geq2$, $\lambda \in V_d$ and $\widehat{\Phi}_d : \mathrm{MC}_d \to \widetilde{\Lambda}_d$, we have
 \[
    \#\left( \widehat{\Phi}_d^{-1}(\bar{\lambda}) \right) 
    = \frac{\sum_{\mathbb{I}\in \mathfrak{I}'(\lambda)} \left( \left\{ -(d-1) \right\}^{\#\mathbb{I} - 1} 
	\cdot \prod_{I \in \mathbb{I}}\left( \#I - 1 \right)! \right)}{\prod_{K \in \mathfrak{K}(\lambda)}\left(\#K\right)!}.
 \]
\end{mcorollary}

\section{Proof  of Theorem~\ref{mainthm1}}\label{proof1}

In this section, we prove Theorem~\ref{mainthm1}.
We assume $d \geq 2$ and $\lambda=(\lambda_1,\dots,\lambda_d) \in V_d$, and denote by $\mathbb{I}_0 = \{\{ 1,\dots,d \}\}$
the minimum element of $\mathfrak{I}'(\lambda)$, which are fixed throughout this section.

First we put
\[
   e_{\mathbb{I}_0}(\lambda) := (d-1)s_d(\lambda)
\]
for $\mathbb{I}_0 = \{\{ 1,\dots,d \}\} \in \mathfrak{I}'(\lambda)$.
Then, equation~(\ref{eq2.2}) for $\mathbb{I} \in \mathfrak{I}(\lambda)$ is rewritten in the form
\begin{equation}\label{eq4.1}
     e_{\mathbb{I}}(\lambda)
   = \prod_{I \in \mathbb{I}} e_{\{I\}}(\lambda_I).
\end{equation}
Here, $\{I\}$ denotes the minimum element of $\mathfrak{I}'(\lambda_I)$.
Moreover, equation~(\ref{eq2.3}) is rewritten in the form
\begin{equation}\label{eq4.2}
     e_{\mathbb{I}_0}(\lambda) = (d-1)!
	  - \sum_{\mathbb{I} \in \mathfrak{I}(\lambda)} \left(
          e_{\mathbb{I}}(\lambda)\cdot
          \prod_{k=d-\# \mathbb{I} +1}^{d-1}k \right),
\end{equation}
which is also equivalent to the equality
\begin{equation*}
 (d-1)! = \sum_{\mathbb{I} \in \mathfrak{I}'(\lambda)} \left(
          e_{\mathbb{I}}(\lambda)\cdot
          \prod_{k=d-\# \mathbb{I} +1}^{d-1}k \right)
\end{equation*}
since for $\mathbb{I}_0 \in \mathfrak{I}'(\lambda)$, we have 
$e_{\mathbb{I}_0}(\lambda)\cdot \prod_{k=d-\# \mathbb{I}_0 +1}^{d-1}k = e_{\mathbb{I}_0}(\lambda)\cdot \prod_{k=d}^{d-1}k = e_{\mathbb{I}_0}(\lambda)$.
Equation~(\ref{eq2.5}), which we would like to prove in this section, is also rewritten in the form
\stepcounter{equation}
\begin{equation}\label{eq4.4}\tag*{$(\theequation)_d$}
    e_{\mathbb{I}_0}(\lambda) = \sum_{\mathbb{I}\in \mathfrak{I}'(\lambda)} 
				\left( \left\{ -(d-1) \right\}^{\#\mathbb{I} - 1} 
				\cdot \prod_{I \in \mathbb{I}}\left( \#I - 1 \right)! \right).
\end{equation}
Hence, to prove Theorem~\ref{mainthm1}, it suffices to derive equation~\ref{eq4.4} from equations~(\ref{eq4.1}) and~(\ref{eq4.2}).

In the following, we show equation~\ref{eq4.4} by induction on $d$.

For $d=2$ or $3$, we have $s_d(\lambda)=1$ and $\mathfrak{I}'(\lambda) = \{ \mathbb{I}_0 \}$ for every $\lambda \in V_d$.
Hence, for $\lambda \in V_d$, we always have 
\[
  e_{\mathbb{I}_0}(\lambda) = (d-1)s_d(\lambda) = d-1
\]
and also have
\begin{align*}
   \sum_{\mathbb{I}\in \mathfrak{I}'(\lambda)} 
				\left( \left\{ -(d-1) \right\}^{\#\mathbb{I} - 1} 
				\cdot \prod_{I \in \mathbb{I}}\left( \#I - 1 \right)! \right)
  &= \left\{ -(d-1) \right\}^{\#\mathbb{I}_0 - 1} \cdot \prod_{I \in \mathbb{I}_0}\left( \#I - 1 \right)!\\
  &= \left\{ -(d-1) \right\}^{1 - 1} \cdot \left( d - 1 \right)! = (d-1)!.
\end{align*}
Since $d-1 = (d-1)!$ for $d=2$ or $3$, we 
have equations~$(\theequation)_2$ and~$(\theequation)_3$.

In the following, we assume $d \geq 4$ and show equation~\ref{eq4.4} by the assumption of
equations~$(\theequation)_2, (\theequation)_3, \dots, (\theequation)_{d-1}$, (\ref{eq4.1}) and~(\ref{eq4.2}).

For each $\mathbb{I} \in \mathfrak{I}(\lambda)$ with $\lambda \in V_d$, 
we put $\mathbb{I}=:\left\{ I_1,\dots,I_l \right\}$. Then,
by using equations~(\ref{eq4.1}) and~$(\theequation)_{d'}$ for $2 \leq d' < d$,
we have the following equalities:
\begin{equation}\label{eq4.5}\begin{split}
  e_{\mathbb{I}}(\lambda)
  &= \prod_{I \in \mathbb{I}} e_{\{I\}}(\lambda_I) 
    = \prod_{u=1}^l e_{\{I_u\}}(\lambda_{I_u}) \\
  &= \prod_{u=1}^l \left(
			\sum_{\mathbb{I}'_u\in \mathfrak{I}'(\lambda_{I_u})} 
			\left[ \left\{ -(\#I_u-1) \right\}^{\#\mathbb{I}'_u - 1} 
			\cdot \prod_{I'_u \in \mathbb{I}'_u}\left( \#I'_u - 1 \right)! \right]
		\right) \\
 &= \sum_{\mathbb{I}'_1\in \mathfrak{I}'(\lambda_{I_1})} \cdots \sum_{\mathbb{I}'_l\in \mathfrak{I}'(\lambda_{I_l})}
	\prod_{u=1}^l \left[
		\left\{ -(\#I_u-1) \right\}^{\#\mathbb{I}'_u - 1} 
			\cdot \prod_{I'_u \in \mathbb{I}'_u}\left( \#I'_u - 1 \right)!
	\right] \\
 &= \sum_\textrm{\scriptsize $\begin{matrix}
				       \mathbb{I}' \in \mathfrak{I}(\lambda) \\
				       \mathbb{I}' \succ \mathbb{I}
				     \end{matrix}$}
	\left[
		\left(
			\prod_{I' \in \mathbb{I}'}\left( \#I' - 1 \right)! 
		\right)
		\cdot
		\left(
			\prod_{u=1}^l \left\{ -(\#I_u-1) \right\}^{\chi_{I_u}\left( \mathbb{I}' \right) - 1}
		\right)
	\right]
\end{split}\end{equation}
since we have the equality
\[
   \left\{ \mathbb{I}'_1 \amalg \cdots \amalg \mathbb{I}'_l \mid 
	\mathbb{I}'_1 \in \mathfrak{I}'\left( \lambda_{I_1} \right), \dots, \mathbb{I}'_l \in \mathfrak{I}'\left( \lambda_{I_l} \right)
   \right\}
  = \left\{ \mathbb{I}' \in \mathfrak{I}(\lambda)  \mid  \mathbb{I}' \succ \mathbb{I}
   \right\}
\]
by definition.
Here, since $\mathbb{I} \succ \mathbb{I}$ holds for $\mathbb{I} \in \mathfrak{I}(\lambda)$, we have
$\mathbb{I} \in \left\{ \mathbb{I}' \in \mathfrak{I}(\lambda)  \mid  \mathbb{I}' \succ \mathbb{I} \right\}$.
Note that in equation~(\ref{eq4.5}), 
$\chi_{I_u}\left( \mathbb{I}' \right) = \#\left( \{ I' \in \mathbb{I}' \mid I' \subseteq I_u \} \right)$ is the function defined 
in Theorem~\ref{thm2.5}. 

Substituting equation~(\ref{eq4.5}) into equation~(\ref{eq4.2}), we have
\begin{equation}\label{eq4.6}\begin{split}
 &e_{\mathbb{I}_0}(\lambda) \\
 &= (d-1)!
	  - \sum_{\mathbb{I} \in \mathfrak{I}(\lambda)} 
			\left\{ 
				\sum_\textrm{\scriptsize $\begin{matrix}
				       \mathbb{I}' \in \mathfrak{I}(\lambda) \\
				       \mathbb{I}' \succ \mathbb{I}
				     \end{matrix}$}
				\left(
					\prod_{I' \in \mathbb{I}'}\left( \#I' - 1 \right)! 
				\right)
				\cdot
				\left(
					\prod_{I \in \mathbb{I}} \left\{ -(\#I-1) \right\}^{\chi_{I}\left( \mathbb{I}' \right) - 1}
				\right)
			\right\}
			\cdot \prod_{k=d-\# \mathbb{I} +1}^{d-1}k 
		\\
 &= (d-1)!
	- \sum_{\mathbb{I}' \in \mathfrak{I}(\lambda)}
		\left\{
			\prod_{I' \in \mathbb{I}'}\left( \#I' - 1 \right)!
		\right\} \cdot
		\left\{
			\sum_\textrm{\scriptsize $\begin{matrix}
						\mathbb{I} \in \mathfrak{I}(\lambda) \\
						\mathbb{I} \prec \mathbb{I}'
						\end{matrix}$}
			\left(
				\prod_{I \in \mathbb{I}} \left\{ -(\#I-1) \right\}^{\chi_{I}\left( \mathbb{I}' \right) - 1}
			\right)
			\cdot
			\prod_{k=d-\# \mathbb{I} +1}^{d-1}k
		\right\}
		.
\end{split}\end{equation}

On the other hand, equation~\ref{eq4.4}, which we would like to prove in this section, is equivalent to the equality
\begin{equation*}
 e_{\mathbb{I}_0}(\lambda) = (d-1)! + \sum_{\mathbb{I}\in \mathfrak{I}(\lambda)} 
				\left( \left\{ -(d-1) \right\}^{\#\mathbb{I} - 1} 
				\cdot \prod_{I \in \mathbb{I}}\left( \#I - 1 \right)! \right),
\end{equation*}
which is also equivalent to
\begin{equation}\label{eq4.7}
 e_{\mathbb{I}_0}(\lambda) = (d-1)! + \sum_{\mathbb{I}'\in \mathfrak{I}(\lambda)} 
				\left[ 
					\left\{\prod_{I' \in \mathbb{I}'}\left( \#I' - 1 \right)!\right\} 
					\cdot 
					\left\{ -(d-1) \right\}^{\#\mathbb{I}' - 1} 
				\right].
\end{equation}
Hence, comparing equations~(\ref{eq4.6}) and~(\ref{eq4.7}), we find that to prove equation~\ref{eq4.4}, 
we only need to show the following equality for each $\mathbb{I}' \in \mathfrak{I}(\lambda)$:
\begin{equation}\label{eq4.8}
 \left\{ -(d-1) \right\}^{\#\mathbb{I}' - 1} 
 =
	- \sum_\textrm{\scriptsize $\begin{matrix}
						\mathbb{I} \in \mathfrak{I}(\lambda) \\
						\mathbb{I} \prec \mathbb{I}'
						\end{matrix}$}
	\left(
		\prod_{I \in \mathbb{I}} \left\{ -(\#I-1) \right\}^{\chi_{I}\left( \mathbb{I}' \right) - 1}
	\right)
	\cdot
	\prod_{k=d-\# \mathbb{I} +1}^{d-1}k.
\end{equation}
Here, equation~(\ref{eq4.8}) is equivalent to the equality
\begin{equation}\label{eq4.9}
	\sum_\textrm{\scriptsize $\begin{matrix}
						\mathbb{I} \in \mathfrak{I}'(\lambda) \\
						\mathbb{I} \prec \mathbb{I}'
						\end{matrix}$}
	\left(
	\prod_{k=d-\# \mathbb{I} +1}^{d-1}k
	\right)
	\cdot
	\left(
		\prod_{I \in \mathbb{I}} \left\{ -(\#I-1) \right\}^{\chi_{I}\left( \mathbb{I}' \right) - 1}
	\right)
	= 0
\end{equation}
since for $\mathbb{I}_0 \in \mathfrak{I}'(\lambda)$ and $\mathbb{I}' \in \mathfrak{I}(\lambda)$, 
we have $\mathbb{I}_0 \prec \mathbb{I}'$ and 
\[
  \left(
  	\prod_{k=d-\# \mathbb{I}_0 +1}^{d-1}k
  \right)
  \cdot
\prod_{I \in \mathbb{I}_0} \left\{ -(\#I-1) \right\}^{\chi_{I}\left( \mathbb{I}' \right) - 1}
  =
  \left( \prod_{k=d}^{d-1}k \right)  \cdot  \left\{ -(d-1) \right\}^{\#\mathbb{I}' - 1}
  = \left\{ -(d-1) \right\}^{\#\mathbb{I}' - 1}.
\]
Hence, to prove Theorem~\ref{mainthm1}, we only need to show equation~(\ref{eq4.9})
for every $d \geq 4$, $\lambda \in V_d$ and 
$\mathbb{I}' \in \mathfrak{I}(\lambda)$.
In the following, instead of expressing $\sum_{\mathbb{I} \in \mathfrak{I}'(\lambda),\, \mathbb{I} \prec \mathbb{I}'}$ 
for $\mathbb{I}' \in \mathfrak{I}(\lambda)$, we simply express $\sum_{\mathbb{I} \prec \mathbb{I}'}$,
because if $\mathbb{I}$ is a partition of $\{1,\dots,d\}$ and 
$\mathbb{I} \prec \mathbb{I}'$ for $\mathbb{I}' \in \mathfrak{I}(\lambda)$,
then we automatically have $\mathbb{I} \in \mathfrak{I}'(\lambda)$.

\vspace{15pt}

To prove equation~(\ref{eq4.9}), we make use of the following.
\begin{definition}\label{def4.1}
 For $\mathbb{I}' \in \mathfrak{I}(\lambda)$ with $\#\mathbb{I}'=l$ and for $k \in \mathbb{Z}$, we put
 \[
    f_{l, k} := \sum_{\mathbb{I} \prec \mathbb{I}',\, \#\mathbb{I}=k} \,
	\prod_{I \in \mathbb{I}} \left\{ -(\#I-1) \right\}^{\chi_{I}\left( \mathbb{I}' \right) - 1}.
 \]
\end{definition}
\begin{remark}
 For $\mathbb{I}' \in \mathfrak{I}(\lambda)$ with $\#\mathbb{I}'=l$ and for $\mathbb{I} \prec \mathbb{I}'$,
 we always have $1 \leq \#\mathbb{I} \leq l$.
 Hence, if $k \leq 0$ or $k \geq l+1$, then we have $f_{l, k}=0$ by definition.
\end{remark}
\begin{example}
 Let us find $f_{l,l}$ and $f_{l,1}$ for $l \geq 2$ in this example.

 Since $\left\{ \mathbb{I} \mid \mathbb{I} \prec \mathbb{I}',\, \#\mathbb{I}=l \right\} = \left\{ \mathbb{I}' \right\}$,
 we have
 \[
	f_{l, l} = \prod_{I \in \mathbb{I}'} \left\{ -(\#I-1) \right\}^{\chi_{I}\left( \mathbb{I}' \right) - 1}
		= \prod_{I \in \mathbb{I}'} \left\{ -(\#I-1) \right\}^{1 - 1}
		= 1,
\]
 whereas
 since $\left\{ \mathbb{I} \mid \mathbb{I} \prec \mathbb{I}',\, \#\mathbb{I}=1 \right\} = \left\{ \mathbb{I}_0 \right\}$,
 we have 
 \[
	f_{l,1} = \prod_{I \in \mathbb{I}_0} \left\{ -(\#I-1) \right\}^{\chi_{I}\left( \mathbb{I}' \right) - 1}
		= \left\{ -(d-1) \right\}^{l - 1}.
 \]
\end{example}
\begin{example}\label{ex4.4}
 Let us also find $f_{4,2}$ in this example.
 For $\mathbb{I}' \in \mathfrak{I}(\lambda)$ with $\#\mathbb{I}'=4$, we can put $\mathbb{I}'=\{ I_1, I_2, I_3, I_4 \}$,
 and in this expression, we have 
 $\left\{ \mathbb{I} \mid \mathbb{I} \prec \mathbb{I}',\, \#\mathbb{I}=2 \right\} = \left\{ \mathbb{I}_1,\dots, \mathbb{I}_7 \right\}$,
 where
 \begin{align*}
     &\mathbb{I}_1=\left\{ I_1,\ I_2\amalg I_3\amalg I_4 \right\},\quad
     \mathbb{I}_2=\left\{ I_2,\ I_1\amalg I_3\amalg I_4 \right\},\\
     &\mathbb{I}_3=\left\{ I_3,\ I_1\amalg I_2\amalg I_4 \right\},\quad
     \mathbb{I}_4=\left\{ I_4,\ I_1\amalg I_2\amalg I_3 \right\},\\
     &\mathbb{I}_5=\left\{ I_1\amalg I_2,\ I_3\amalg I_4 \right\},\quad
     \mathbb{I}_6=\left\{ I_1\amalg I_3,\ I_2\amalg I_4 \right\}\quad \text{and} \quad
     \mathbb{I}_7=\left\{ I_1\amalg I_4,\ I_2\amalg I_3 \right\}.
 \end{align*}
 We put $\#I_u=:i_u$ for $1 \leq u \leq 4$. Note that the equality $i_1 + i_2 + i_3 + i_4 = d$ holds.
 We have
 \begin{align*}
     &\prod_{I \in \mathbb{I}_1} \left\{ -(\#I-1) \right\}^{\chi_{I}\left( \mathbb{I}' \right) - 1}
     = \left\{ -(i_1-1) \right\}^{1-1} \cdot \left\{ -(i_2+i_3+i_4-1) \right\}^{3-1},\\
     &\prod_{I \in \mathbb{I}_5} \left\{ -(\#I-1) \right\}^{\chi_{I}\left( \mathbb{I}' \right) - 1}
     = \left\{ -(i_1+i_2-1) \right\}^{2-1} \cdot \left\{ -(i_3+i_4-1) \right\}^{2-1}
 \end{align*}
 for instance, which implies
 \begin{align*}
     \sum_{u=1}^4 \prod_{I \in \mathbb{I}_u} \left\{ -(\#I-1) \right\}^{\chi_{I}\left( \mathbb{I}' \right) - 1}
     &= \sum_{u=1}^4 (i_1+i_2+i_3+i_4 - i_u-1)^2
     = \sum_{u=1}^4 (d - i_u-1)^2\\
     &= 4(d-1)^2 - 2(d-1)d  + \sum_{u=1}^4 i_u^2,\\
     \sum_{u=5}^7 \prod_{I \in \mathbb{I}_u} \left\{ -(\#I-1) \right\}^{\chi_{I}\left( \mathbb{I}' \right) - 1}
     &= (i_1+i_2-1)(i_3+i_4-1) + (i_1+i_3-1)(i_2+i_4-1) \\ &\hspace*{13pt}+ (i_1+i_4-1)(i_2+i_3-1)
     = 2\sum_{1 \leq u <v \leq 4}i_ui_v - 3d + 3.
 \end{align*}
 Hence, we have 
 \begin{align*}
   f_{4, 2} &= \sum_{u=1}^7 \prod_{I \in \mathbb{I}_u} \left\{ -(\#I-1) \right\}^{\chi_{I}\left( \mathbb{I}' \right) - 1}\\
	&= 4(d-1)^2 - 2(d-1)d  + \sum_{u=1}^4 i_u^2 + 2\sum_{1 \leq u <v \leq 4}i_ui_v - 3d + 3\\
       &= 2d^2 - 9d + 7 + \left(\sum_{u=1}^4 i_u \right)^2 = 3d^2 - 9d + 7.
 \end{align*}
\end{example}
\begin{example}
 By a similar computation to Example~\ref{ex4.4}, we have the following for $l \leq 5$:
 \begin{align*}
   f_{2,1} &= -d+1, & f_{3,1} &= (d-1)^2, & f_{4,1} &= \{-(d-1)\}^3, & f_{5,1} &= \{-(d-1)\}^4,\\
   f_{2,2} &= 1,      & f_{3,2} &= -2d+3,  & f_{4,2} &= 3d^2-9d+7,   & f_{5,2} &= -4d^3+18d^2-28d+15,\\
   	    &		  &  f_{3,3} &= 1,		& f_{4,3} &= -3d+6,	      & f_{5,3} &= 6d^2-24d+25,\\
	    &		  &		&		& f_{4,4} &= 1,		      & f_{5,4} &= -4d+10,\\
	    &		  &		&		&	    &			      & f_{5,5} &= 1.
 \end{align*}
\end{example}

The following is the key proposition to prove equation~(\ref{eq4.9}).
\begin{proposition}\label{prop4.6}
 The number $f_{l, k}$ defined in Definition~\ref{def4.1} is a function of $l, k$ and $d$, 
 and does not depend on the choice of $\mathbb{I}' \in \mathfrak{I}(\lambda)$
 with $\#\mathbb{I}'=l$.
 Moreover, for $l, k \in \mathbb{Z}$ with $l \geq 2$, we have the equality
 \[
     f_{l+1, k} = f_{l,k-1} - (d-k)f_{l,k}.
 \]
\end{proposition}

\begin{proposition}\label{prop4.7}
 Admitting Proposition~\ref{prop4.6}, 
 we have equation~(\ref{eq4.9}) for every $d \geq 4$, $\lambda \in V_d$ and 
 $\mathbb{I}' \in \mathfrak{I}(\lambda)$.
 Hence, Proposition~\ref{prop4.6} implies Theorem~\ref{mainthm1}.
\end{proposition}

\begin{proof}[Proof of Proposition~\ref{prop4.7}]
 If $\#\mathbb{I}'=2$, then we can put $\mathbb{I}'=\{I_1, I_2\}$ and
 have $\left\{ \mathbb{I} \mid \mathbb{I} \prec \mathbb{I}' \right\} = \{\mathbb{I}_0, \mathbb{I}'\}$.
 Hence, we have
 \begin{align*}
     &\sum_{\mathbb{I} \prec \mathbb{I}'}
	\left(
	\prod_{k=d-\# \mathbb{I} +1}^{d-1}k
	\right)
	\cdot
	\left(
		\prod_{I \in \mathbb{I}} \left\{ -(\#I-1) \right\}^{\chi_{I}\left( \mathbb{I}' \right) - 1}
	\right)\\
     &= 1\cdot \{ -(d-1) \}^{2-1} + (d-1) \cdot \{ -(\#I_1-1) \}^{1-1} \cdot \{ -(\#I_2-1) \}^{1-1}\\
     &= -(d-1) + (d-1) = 0.
 \end{align*}

 In the case where $\#\mathbb{I}' \geq 3$, we put $\#\mathbb{I}' =:l+1$.
 Then we have $l \geq 2$ and have the following equalities by Proposition~\ref{prop4.6}:
 \begin{align*}
  &\sum_{\mathbb{I} \prec \mathbb{I}'}
	\left(
		\prod_{k=d-\# \mathbb{I} +1}^{d-1}k
	\right)
	\cdot
	\left(
		\prod_{I \in \mathbb{I}} \left\{ -(\#I-1) \right\}^{\chi_{I}\left( \mathbb{I}' \right) - 1}
	\right)\\
  &= \sum_{k=1}^{l+1}
	\left(
		\prod_{k'=d-k +1}^{d-1}k'
	\right)
	\cdot f_{l+1, k}\\
  &= \sum_{k=1}^{l+1}
	\left(
		\prod_{k'=d-k +1}^{d-1}k'
	\right)
	\cdot 
	\Bigl(
		f_{l, k-1} - (d-k)f_{l,k}
	\Bigr)\\
  &=  \sum_{k=1}^{l+1}
	\left(
		\prod_{k'=d-k +1}^{d-1}k'
	\right)
	\cdot f_{l, k-1}
	-
	\sum_{k=1}^{l+1}
	\left(
		\prod_{k'=d-k +1}^{d-1}k'
	\right)
	\cdot (d-k)f_{l,k}\\
  &=  \sum_{k=0}^{l}
	\left(
		\prod_{k'=d-k}^{d-1}k'
	\right)
	\cdot f_{l, k}
	-
	\sum_{k=1}^{l+1}
	\left(
		\prod_{k'=d-k}^{d-1}k'
	\right)
	\cdot f_{l,k}\\
  &=  \left(
		\prod_{k'=d}^{d-1}k'
	\right)
	\cdot f_{l, 0}
	-
	\left(
		\prod_{k'=d-(l+1)}^{d-1}k'
	\right)
	\cdot f_{l,l+1}
	 = 0,
 \end{align*}
 which completes the proof of Proposition~\ref{prop4.7}.
\end{proof}

\vspace{15pt}

In the rest of this section, we shall prove Proposition~\ref{prop4.6}.
We make use of the following polynomial to prove Proposition~\ref{prop4.6}.
\begin{definition}\label{def4.8}
 For $l, k \in \mathbb{Z}$ with $l \geq 2$, we define $\mathfrak{J}_l(k)$ as follows:
 if $k \leq 0$ or $k \geq l+1$, then we put $\mathfrak{J}_l(k) = \emptyset$; 
 if $1 \leq k \leq l$, then we put
 \[
     \mathfrak{J}_l(k) := \left\{ \left\{J_1,\ldots,J_k\right\}\ \left| \ 
	 \begin{matrix}
	   J_1 \amalg \cdots \amalg J_k = \{1,\ldots,l \},\\
	   J_v \ne \emptyset \textrm{ for every } 1\le v\le k
	 \end{matrix}
	 \right.\right\},
 \]
 where $J_1 \amalg \cdots \amalg J_k$ denotes the disjoint union of $J_1,\dots, J_k$.
 Moreover, for $l, k \in \mathbb{Z}$ with $l \geq 2$, we put
 \[
     g_{l,k}(X_1,\dots,X_l) := \sum_{\mathbb{J} \in \mathfrak{J}_l(k)} \prod_{J\in \mathbb{J}}
	\left\{
		- \left(
			\sum_{u\in J} X_u - 1
		\right)
	\right\}^{\#J-1}.
 \]
\end{definition}
By definition, $\mathfrak{J}_l(k)$ is the set of all the partitions of $\{1, \dots,l\}$ into $k$ pieces.
Note that the equality $g_{l,k}(X_1,\dots,X_l)=0$ trivially holds for $k \leq 0$ or $k \geq l+1$.

\begin{lemma}\label{lem4.9}
 For $\mathbb{I}' \in \mathfrak{I}(\lambda)$ with $\#\mathbb{I}'=l$ and for every $k \in \mathbb{Z}$, 
 putting $\mathbb{I}' =: \left\{ I_1, \dots, I_l \right\}$ and $\# I_u =: i_u$ for $1 \leq u \leq l$,
 we have
 \begin{equation}\label{eq4.10}
    f_{l,k} = g_{l,k}(i_1,\dots,i_l).
 \end{equation}
\end{lemma}
\begin{proof}
 If $k \leq 0$ or $k \geq l+1$, then equation~(\ref{eq4.10}) trivially holds 
 since both sides of equation~(\ref{eq4.10}) are equal to zero.
 In the following, we assume $1 \leq k \leq l$.

 By definition, we have
 \[
    f_{l,k} = \sum_{\mathbb{I} \prec \mathbb{I}',\, \#\mathbb{I}=k} \,
	\prod_{I \in \mathbb{I}} \left\{ -(\#I-1) \right\}^{\chi_{I}\left( \mathbb{I}' \right) - 1}
	= \sum_{\mathbb{I} \prec \mathbb{I}',\, \#\mathbb{I}=k} \,
		\prod_{I \in \mathbb{I}} 
		\left\{ -\left(
			\sum_{1 \leq u \leq l,\ I_u \subset I} i_u-1
		\right) \right\}^{\chi_{I}\left( \mathbb{I}' \right) - 1}.
 \]
 Hence, putting
 \[
     \tilde{g}_{l,k}(X_1,\dots,X_l) 
     := \sum_{\mathbb{I} \prec \mathbb{I}',\, \#\mathbb{I}=k} \,
		\prod_{I \in \mathbb{I}} 
		\left\{ -\left(
			\sum_{1 \leq u \leq l,\ I_u \subset I} X_u-1
		\right) \right\}^{\chi_{I}\left( \mathbb{I}' \right) - 1},
 \]
 we obviously have 
 $\tilde{g}_{l,k}(i_1,\dots,i_l) = f_{l,k}$.

 On the other hand, we can make a bijection 
 $\mathfrak{J}_l(k) \to \{\mathbb{I} \mid \mathbb{I} \prec \mathbb{I}',\, \#\mathbb{I}=k \}$ by
 \[
    \mathbb{J} \mapsto \left\{\left. \amalg_{u \in J} I_u \right| \ J \in \mathbb{J} \right\},
 \]
 which implies that
 \begin{align*}
  \tilde{g}_{l,k}(X_1,\dots,X_l) 
	&= \sum_{\mathbb{J} \in \mathfrak{J}_l(k)} \,
		\prod_{I \in \left\{\left. \amalg_{u \in J} I_u \right| \ J \in \mathbb{J} \right\}} 
		\left\{ -\left(
			\sum_{1 \leq u \leq l,\ I_u \subset I} X_u-1
		\right) \right\}^{\chi_{I}\left( \mathbb{I}' \right) - 1}\\
	&= \sum_{\mathbb{J} \in \mathfrak{J}_l(k)} \,
		\prod_{J \in \mathbb{J}} 
		\left\{ -\left(
			\sum_{1 \leq u \leq l,\ I_u \subset \amalg_{u' \in J} I_{u'}} X_u - 1
		\right) \right\}^{\chi_{\left(\amalg_{u' \in J} I_{u'}\right)}\left( \mathbb{I}' \right) - 1}\\
	&= \sum_{\mathbb{J} \in \mathfrak{J}_l(k)} \,
		\prod_{J \in \mathbb{J}} 
		\left\{ -\left(
			\sum_{u \in J} X_u - 1
		\right) \right\}^{\#J - 1}
	  = g_{l,k}(X_1,\dots,X_l).
 \end{align*}
 Hence, we have equation~(\ref{eq4.10}).
\end{proof}

\begin{lemma}\label{lem4.10}
 The polynomial $g_{l,k}(X_1,\dots,X_l)$ defined in Definition~\ref{def4.8}
 is determined only by $l$ and $k$,
 belongs to the polynomial ring $\mathbb{Z}[X_1,\dots,X_l]$,
 and is symmetric in $l$ variables $X_1,\dots,X_l$.
 Moreover, the equality~$\deg g_{l,k} = l-k$ holds for $l \geq 2$ and $1 \leq k \leq l$.
\end{lemma}

\begin{proof}
 The former two assertions are obvious by definition.

 The action of $\mathfrak{S}_l$ on $\{1,\dots,l\}$ naturally induces 
 the action of $\mathfrak{S}_l$ on $\mathfrak{J}_l(k)$ for each $k$, 
 which implies that 
 for every $\tau \in \mathfrak{S}_l$, we have $g_{l,k}\left(X_{\tau(1)},\dots,X_{\tau(l)}\right) = g_{l,k}\left(X_1,\dots,X_l\right)$.
 Hence, $g_{l,k}(X_1,\dots,X_l)$ is a symmetric polynomial in $l$ variables $X_1,\dots,X_l$.

 Since $\sum_{J \in \mathbb{J}} \left( \#J-1 \right) = l-\#\mathbb{J} = l-k$ for every $\mathbb{J} \in \mathfrak{J}_l(k)$,
 we have $\deg g_{l,k} \leq l-k$.
 Moreover, for $\mathbb{J} \in \mathfrak{J}_l(k)$ with $1 \leq k \leq l$, the coefficient of each term of 
	$\prod_{J\in \mathbb{J}}
	\left\{
		- \left(
			\sum_{u\in J} X_u - 1
		\right)
	\right\}^{\#J-1}$
 with degree $l-k$ is positive or negative according to whether $l-k$ is even or odd.
 Hence, the terms with degree $l-k$ in $g_{l,k}\left(X_1,\dots,X_l\right)$ are not canceled, 
 which implies that
 the degree of $g_{l,k}(X_1,\dots,X_l)$ is exactly equal to $l-k$ if $1 \leq k \leq l$.
\end{proof}

\begin{proposition}\label{prop4.11}
 For $l, k \in \mathbb{Z}$ with $l \geq 2$, we have
 \[
     g_{l+1, k}(X_1, \dots, X_l, 0) = g_{l, k-1}(X_1,\dots, X_l) - \left(X_1+\dots+X_l - k\right) g_{l,k}(X_1,\dots, X_l).
 \]
\end{proposition}
\begin{proof}
 First we put 
 \[
	\mathfrak{J}_{l+1}^1(k) := \left\{ \mathbb{J} \in \mathfrak{J}_{l+1}(k) \mid \{l+1\} \in \mathbb{J} \right\} 
	\quad \text{and} \quad
	\mathfrak{J}_{l+1}^2(k) := \left\{ \mathbb{J} \in \mathfrak{J}_{l+1}(k) \mid \{l+1\} \notin \mathbb{J} \right\}
 \]
 for $l \geq 2$.  Then we have $\mathfrak{J}_{l+1}^1(k) \, \amalg \, \mathfrak{J}_{l+1}^2(k) = \mathfrak{J}_{l+1}(k)$ for every $k$.
 Moreover, we have $\mathfrak{J}_{l+1}^1(k)= \emptyset$ for $k \leq 1$ or $k \geq l+2$, and 
 $\mathfrak{J}_{l+1}^2(k)= \emptyset$ for $k \leq 0$ or $k \geq l+1$.

 For $\mathbb{J} \in \mathfrak{J}_{l+1}^1(k)$, we can express $\mathbb{J} = \{J_1,\dots,J_{k-1}, \{l+1\}\}$, where
 $J_1 \amalg \dots \amalg J_{k-1} = \{1,\dots,l\}$.
 Hence, we can make a bijection $\pi_1: \mathfrak{J}_{l+1}^1(k) \to \mathfrak{J}_{l}(k-1)$ by
 $\mathbb{J} \mapsto \mathbb{J} \setminus \{\{l+1\}\}$.
 Moreover, for $J=\{l+1\} \in \mathbb{J} \in \mathfrak{J}_{l+1}^1(k)$, we have
 \[
	\left\{
		- \left(
			\sum_{u\in J} X_u - 1
		\right)
	\right\}^{\#J-1}
	=
	\left\{
		- \left(
			X_{l+1} - 1
		\right)
	\right\}^{1-1}
	= 1.
 \]
 Hence, we have
 \begin{equation}\label{eq4.11}\begin{split}
	\sum_{\mathbb{J} \in \mathfrak{J}_{l+1}^1(k)} \prod_{J\in \mathbb{J}}
	\left\{
		- \left(
			\sum_{u\in J} X_u - 1
		\right)
	\right\}^{\#J-1}
	&=
	\sum_{\mathbb{J} \in \mathfrak{J}_{l+1}^1(k)} \prod_{J\in \pi_1\left(\mathbb{J}\right)}
	\left\{
		- \left(
			\sum_{u\in J} X_u - 1
		\right)
	\right\}^{\#J-1}\\
	&=
	\sum_{\mathbb{J} \in \mathfrak{J}_{l}(k-1)} \prod_{J\in \mathbb{J}}
	\left\{
		- \left(
			\sum_{u\in J} X_u - 1
		\right)
	\right\}^{\#J-1}\\
	&=
	g_{l,k-1}(X_1,\dots,X_l).
 \end{split}\end{equation}

 For $\mathbb{J}' \in \mathfrak{J}_{l+1}^2(k)$, we can express $\mathbb{J}' = \{J_1,\dots,J_k\}$ with $\{l+1\} \subsetneq J_k$,
 and in this expression, we have $\{J_1,\dots,J_{k-1}, \left(J_k\setminus\{l+1\}\right) \} \in \mathfrak{J}_{l}(k)$.
 Hence, we can make a surjection $\pi_2: \mathfrak{J}_{l+1}^2(k) \to \mathfrak{J}_{l}(k)$ by 
 $\mathbb{J}' \mapsto \left\{ J \setminus \{l+1\} \mid J \in \mathbb{J}' \right\}$.
 For each $\mathbb{J} = \{J_1,\dots,J_k\} \in \mathfrak{J}_{l}(k)$, 
 its fiber $\pi_2^{-1}\left( \mathbb{J} \right)$ consists of $k$ elements,
 which are $\{J_v \mid 1 \leq v \leq k,\ v \ne v' \} \cup \{ J_{v'}\amalg\{l+1\} \}$ for $1 \leq v' \leq k$.
 Hence, for each $\mathbb{J} = \{J_1,\dots,J_k\} \in \mathfrak{J}_{l}(k)$, we have
 \begin{align*}
	&\left.
	\sum_{\mathbb{J}' \in \pi_2^{-1}\left( \mathbb{J} \right)} \prod_{J\in \mathbb{J}'}
	\left\{
		- \left(
			\sum_{u\in J} X_u - 1
		\right)
	\right\}^{\#J-1}
	\right|_{X_{l+1}=0}\\
	&= 
	\sum_{v'=1}^k 
	\left[
		\left\{
			-\left(
				\sum_{u \in J_{v'}\amalg\{l+1\}} X_u -1
			\right)
		\right\}^{\#\left( J_{v'}\amalg \{l+1\} \right) - 1} \right.\\
		&\hspace{161pt}\left.\left.\times 
		\prod_{1 \leq v \leq k,\ v \ne v'}
		\left\{
			-\left(
				\sum_{u \in J_v} X_u - 1
			\right)
		\right\}^{\#J_v - 1}
	\right]
	\right|_{X_{l+1}=0}\\
	&= \sum_{v'=1}^k \left[
		\left\{
			-\left(
				\sum_{u \in J_{v'}} X_u -1
			\right)
		\right\}^{\# J_{v'}}
		\cdot 
		\prod_{1 \leq v \leq k,\ v \ne v'}
		\left\{
			-\left(
				\sum_{u \in J_v} X_u - 1
			\right)
		\right\}^{\#J_v - 1}
	\right]\\
	&= \sum_{v'=1}^k \left[
		\left\{
			-\left(
				\sum_{u \in J_{v'}} X_u -1
			\right)
		\right\}
		\cdot 
		\prod_{v=1}^k
		\left\{
			-\left(
				\sum_{u \in J_v} X_u - 1
			\right)
		\right\}^{\#J_v - 1}
	\right]\\
	&= \left[
		\sum_{v'=1}^k 
		\left\{
			-\left(
				\sum_{u \in J_{v'}} X_u -1
			\right)
		\right\}
	\right]
		\cdot 
		\prod_{v=1}^k
		\left\{
			-\left(
				\sum_{u \in J_v} X_u - 1
			\right)
		\right\}^{\#J_v - 1}\\
	&= - \left(
			\sum_{u=1}^l X_u - k
		\right)
		\cdot 
		\prod_{J \in \mathbb{J}}
		\left\{
			-\left(
				\sum_{u \in J} X_u - 1
			\right)
		\right\}^{\#J - 1}.
 \end{align*}
 We therefore have
 \begin{equation}\label{eq4.12}\begin{split}
	&\left.
	\sum_{\mathbb{J}' \in \mathfrak{J}_{l+1}^2(k)} \prod_{J\in \mathbb{J}'}
	\left\{
		- \left(
			\sum_{u\in J} X_u - 1
		\right)
	\right\}^{\#J-1}
	\right|_{X_{l+1}=0}\\
	&= \left.
	\sum_{\mathbb{J} \in \mathfrak{J}_l(k)} 
	\sum_{\mathbb{J}' \in \pi_2^{-1}\left( \mathbb{J} \right)} \prod_{J\in \mathbb{J}'}
	\left\{
		- \left(
			\sum_{u\in J} X_u - 1
		\right)
	\right\}^{\#J-1}
	\right|_{X_{l+1}=0}\\
	&= \sum_{\mathbb{J} \in \mathfrak{J}_l(k)}
	\left[
		- \left(
			\sum_{u=1}^l X_u - k
		\right)
		\cdot 
		\prod_{J \in \mathbb{J}}
		\left\{
			-\left(
				\sum_{u \in J} X_u - 1
			\right)
		\right\}^{\#J - 1}
	\right]\\
	&= - \left(
			\sum_{u=1}^l X_u - k
		\right)
		\sum_{\mathbb{J} \in \mathfrak{J}_l(k)}
		\prod_{J \in \mathbb{J}}
		\left\{
			-\left(
				\sum_{u \in J} X_u - 1
			\right)
		\right\}^{\#J - 1}\\
	&= - \left(X_1+\dots+X_l-k\right)g_{l,k}(X_1,\dots,X_l).
 \end{split}\end{equation}

 By equations~(\ref{eq4.11}) and~(\ref{eq4.12}), we have
 \begin{align*}
	g_{l+1,k}(X_1,\dots,X_l,0) 
	&= \left.
	\sum_{\mathbb{J} \in \mathfrak{J}_{l+1}(k)} \prod_{J\in \mathbb{J}}
	\left\{
		- \left(
			\sum_{u\in J} X_u - 1
		\right)
	\right\}^{\#J-1}
	\right|_{X_{l+1}=0}\\
	&= \sum_{\mathbb{J} \in \mathfrak{J}_{l+1}^1(k)} \prod_{J\in \mathbb{J}}
	\left\{
		- \left(
			\sum_{u\in J} X_u - 1
		\right)
	\right\}^{\#J-1}\\
	&\hspace*{13pt}+
	\left.
	\sum_{\mathbb{J}' \in \mathfrak{J}_{l+1}^2(k)} \prod_{J\in \mathbb{J}'}
	\left\{
		- \left(
			\sum_{u\in J} X_u - 1
		\right)
	\right\}^{\#J-1}
	\right|_{X_{l+1}=0}\\
	&= g_{l,k-1}(X_1,\dots,X_l)  - \left(X_1+\dots+X_l-k\right)g_{l,k}(X_1,\dots,X_l),
 \end{align*}
 which completes the proof of Proposition~\ref{prop4.11}.
\end{proof}

\begin{lemma}\label{lem4.12}
 For every $l, k \in \mathbb{Z}$ with $l \geq 2$, there exists a polynomial $h_{l,k}(Y) \in \mathbb{Z}[Y]$ 
 such that the equality 
 \begin{equation}\label{eq4.13}
	g_{l,k}(X_1,\dots,X_l)= h_{l,k}(X_1 + \dots + X_l)
 \end{equation}
 holds.
 Moreover, for every $l, k \in \mathbb{Z}$ with $l \geq 2$, the equality 
 \begin{equation}\label{eq4.14}
    h_{l+1,k}(Y) = h_{l,k-1}(Y) - (Y-k)h_{l,k}(Y)
 \end{equation}
 holds.
\end{lemma}
\begin{proof}
 In the case where $l=2$, we have $g_{2,1}(X_1, X_2) = -(X_1 + X_2 - 1)$ and $g_{2,2}(X_1, X_2) = 1$
 by a direct calculation.
 Hence, putting $h_{2,1}(Y) = -(Y-1)$, $h_{2,2}(Y) = 1$, and $h_{2,k}(Y) = 0$ for $k \ne 1, 2$, 
 we have $g_{2,k}(X_1, X_2) = h_{2, k}(X_1+X_2)$ for every $k \in \mathbb{Z}$.

 For $l \geq 3$ and for every $k\in\mathbb{Z}$, we define the polynomials $h_{l,k}(Y)$ inductively by equation~(\ref{eq4.14}).
 Then we obviously have $h_{l,k}(Y)=0$ for $k \leq 0$ or $k \geq l+1$.
 Hence, equation~(\ref{eq4.13}) holds for $k \leq 0$ or $k \geq l+1$.
 In the following, we show equation~(\ref{eq4.13}) for $l\geq 3$ and $1 \leq k \leq l$ by induction on $l$.
 Hence, we suppose equation~(\ref{eq4.13}) for every $k \in \mathbb{Z}$,
 and show the equality $g_{l+1,k}(X_1,\dots,X_{l+1})= h_{l+1,k}(X_1 + \dots + X_{l+1})$ for $1 \leq k \leq l+1$.

 By the assumption and Proposition~\ref{prop4.11}, we have
 \begin{align*}
	g_{l+1, k}(X_1, \dots, X_l, 0) 
	&= g_{l, k-1}(X_1,\dots, X_l) - \left(X_1+\dots+X_l - k\right) g_{l,k}(X_1,\dots, X_l)\\
	&= h_{l, k-1}(X_1+\dots+X_l) - \left(X_1+\dots+X_l - k\right) h_{l,k}(X_1+\dots+X_l)\\
	&= h_{l+1,k}(X_1+\dots+X_l).
 \end{align*}
 Hence, putting $P_{l+1,k}(X_1,\dots,X_{l+1}) := g_{l+1, k}(X_1, \dots, X_{l+1}) - h_{l+1,k}(X_1+\dots+X_{l+1})$,
 we have $P_{l+1,k}(X_1,\dots,X_l,0)=0$.
 Moreover, by Lemma~\ref{lem4.10}, the polynomial $P_{l+1,k}(X_1,\dots,X_{l+1})$ is symmetric
 in $l+1$ variables $X_1,\dots,X_{l+1}$.

 We denote by $\sigma_{l+1, m}=\sigma_{l+1,m}(X_1,\dots,X_{l+1})$ the elementary symmetric polynomial of degree $m$
 in $l+1$ variables $X_1,\dots,X_{l+1}$.
 Since $P_{l+1,k}(X_1,\dots,X_{l+1})$ is a symmetric polynomial with coefficients in $\mathbb{Z}$,
 we have $P_{l+1,k}(X_1,\dots,X_{l+1}) \in \mathbb{Z}\left[ \sigma_{l+1,1},\dots,\sigma_{l+1,l+1} \right]$.
 Moreover, since $\deg g_{l+1,k} = \deg h_{l+1,k} = l+1-k \leq l$, we have $\deg P_{l+1,k} \leq l$,
 which implies that $P_{l+1,k}(X_1,\dots,X_{l+1}) \in \mathbb{Z}\left[ \sigma_{l+1,1},\dots,\sigma_{l+1,l} \right]$.

 Since $\sigma_{l+1,m}(X_1,\dots,X_l,0) = \sigma_{l,m}(X_1,\dots,X_l)$ for $1 \leq m \leq l$,
 we have a ring isomorphism 
 $\varphi : \mathbb{Z}\left[ \sigma_{l+1,1},\dots,\sigma_{l+1,l} \right] \to \mathbb{Z}\left[ \sigma_{l,1},\dots,\sigma_{l,l} \right]$
 by substituting $X_{l+1}=0$,
 and under the map $\varphi$,
 we have $\varphi\left( P_{l+1,k} \right) =
 P_{l+1,k}(X_1,\dots,X_l,0)=0$.
 Hence, injectivity of $\varphi$ implies $P_{l+1,k}(X_1,\dots,X_{l+1})=0$.
 We therefore have $g_{l+1,k}(X_1,\dots,X_{l+1})= h_{l+1,k}(X_1 + \dots + X_{l+1})$,
 which completes the proof of Lemma~\ref{lem4.12} by induction on $l$.
\end{proof}

\begin{proof}[Proof of Proposition~\ref{prop4.6}]
 By Definition~\ref{def4.1}, $f_{l,k}$ is originally a function of $d \geq4$, $\mathbb{I}' \in \mathfrak{I}(\lambda)$ and $k \in \mathbb{Z}$.
 However, putting $\#\mathbb{I}'=l$, $\mathbb{I}' =: \left\{ I_1, \dots, I_l \right\}$ and $\# I_u =: i_u$ for $1 \leq u \leq l$,
 we have, by Lemmas~\ref{lem4.9} and~\ref{lem4.12}, the equality
 \begin{equation}\label{eq4.15}
    f_{l,k} = g_{l,k}(i_1,\dots,i_l)=h_{l,k}(i_1+\dots+i_l) = h_{l,k}(d).
 \end{equation}
 Hence, $f_{l,k}$ is in practice a function of $l, k$ and $d$
 since the polynomial $h_{l,k}(Y)$ depends only on $l$ and $k$.

 Moreover, by equation~(\ref{eq4.15}) and Lemma~\ref{lem4.12}, we have
 \[
     f_{l+1,k} = h_{l+1,k}(d) = h_{l,k-1}(d) - (d-k)h_{l,k}(d) = f_{l,k-1} - (d-k)f_{l,k}
 \]
 for every $l,k\in\mathbb{Z}$ with $l \geq 2$,
 which completes the proof of Proposition~\ref{prop4.6}.
\end{proof}

To summarize the above mentioned, we have completed the proof of Theorem~\ref{mainthm1}.

\section{Proof  of Theorem~\ref{mainthm2}}\label{proof2}

In this section, we prove Theorem~\ref{mainthm2}.
Throughout this section, we always assume $\lambda=(\lambda_1,\dots,\lambda_d) \in V_d$,
and moreover assume that $s_d(\lambda)$ is the non-negative integer defined in Theorem~\ref{thm2.3}.

First, we consider the case where $d=2$.
If $d=2$, then the maps $p: \mathrm{MC}_2 \to \mathrm{MP}_2$ and $\Phi_2: \mathrm{MP}_2 \to \widetilde{\Lambda}_2$ are bijective.
Hence, we have $\#\left( \widehat{\Phi}_2^{-1} \left( \bar{\lambda} \right) \right) = 1$
for every $\lambda \in V_2$.
On the other hand, since $s_2(\lambda)= 1$ and $\mathfrak{K}(\lambda) = \left\{ \{1\}, \{2\} \right\}$ for every $\lambda \in V_2$, we always have
\[
    \frac{(d-1)s_d(\lambda)}{ \prod_{K\in\mathfrak{K}(\lambda)}\left(\#K\right)!}
    = \frac{(2-1)s_2(\lambda)}{ 1!\cdot 1!} = 1.
\]
Hence, equation~(\ref{eq3.2}) holds for every $\lambda \in V_2$.

In the rest of this section, we consider the case where $d \geq 3$.
We denote by $\mathbb{P}^{d-1}$ the complex projective space of dimension~$d-1$, and put
\[
 \Sigma_d(\lambda) := \left\{(\zeta_1:\cdots:\zeta_d) \in \mathbb{P}^{d-1} \
        \left|\ \begin{matrix}
		 \sum_{i=1}^d \zeta_i = 0 \\
		 \sum_{i=1}^d \frac{1}{1-\lambda_i}\zeta_i^k = 0 \quad 
		  \textrm{for} \quad 1 \le k \le d-2 \\
		 \zeta_1,\ldots,\zeta_d \textrm{ are mutually distinct}
		\end{matrix}
        \right. \right\}.
\]
We already have the following proposition by Propositions~4.3 and~9.1 in~\cite{sugi1}.

\begin{proposition}\label{prop5.1}
 The equality $\#\left( \Sigma_d(\lambda) \right) = s_d(\lambda)$ holds.
 Moreover, we can define the surjection $\pi(\lambda) : \Sigma_d(\lambda) \to \Phi_d^{-1}(\bar{\lambda})$ by 
 \[
    (\zeta_1:\cdots:\zeta_d) \mapsto f(z)=z + \rho (z - \zeta_1)\cdots(z-\zeta_d),
 \]
 where $-\frac{1}{\rho}= \sum_{i=1}^d \frac{1}{1-\lambda_i}\zeta_i^{d-1}$.
\end{proposition}

We put
\[
 \widetilde{\Sigma}_d(\lambda) := \left\{(\zeta_1,\dots,\zeta_d) \in \mathbb{C}^{d} \
        \left|\ \begin{matrix}
		 \sum_{i=1}^d \zeta_i = 0 \\
		 \sum_{i=1}^d \frac{1}{1-\lambda_i}\zeta_i^k 
			= \begin{cases} 0 & \textrm{for} \quad 1 \le k \le d-2 \\
				-1 & \textrm{for} \quad k = d-1 \end{cases}\\
		 \zeta_1,\ldots,\zeta_d \textrm{ are mutually distinct}
		\end{matrix}
        \right. \right\}.
\]
Then the natural projection $\widetilde{\Sigma}_d(\lambda) \to \Sigma_d(\lambda)$ 
defined by $(\zeta_1,\dots,\zeta_d) \mapsto (\zeta_1:\cdots:\zeta_d)$
is a $(d-1)$-to-one map because for every $(\zeta_1:\cdots:\zeta_d) \in \Sigma_d(\lambda)$, 
we have $\sum_{i=1}^d \frac{1}{1-\lambda_i}\zeta_i^{d-1} \ne 0$ by Proposition~\ref{prop5.1}.
Hence, we have
\begin{equation}\label{eq5.1}
\#\left( \widetilde{\Sigma}_d(\lambda) \right) = (d-1) \#\left( \Sigma_d(\lambda) \right) = (d-1)s_d(\lambda).
\end{equation}

We consider next the relation between $\widetilde{\Sigma}_d(\lambda)$ and $\widehat{\Phi}_d^{-1}(\bar{\lambda})$.
We can define the surjection 
$\widehat{\pi}(\lambda) : \widetilde{\Sigma}_d(\lambda) \to \widehat{\Phi}_d^{-1}(\bar{\lambda})$ by
\[
   (\zeta_1,\dots,\zeta_d) \mapsto f(z)=z+(z - \zeta_1)\cdots(z-\zeta_d)
\]
by lifting up the map $\pi(\lambda) : \Sigma_d(\lambda) \to \Phi_d^{-1}(\bar{\lambda})$ in Proposition~\ref{prop5.1}.
Here, since $d \geq 3$, every polynomial $f(z)=z+(z - \zeta_1)\cdots(z-\zeta_d)$ for $(\zeta_1,\dots,\zeta_d) \in \widetilde{\Sigma}_d(\lambda)$
is monic and centered.

We put
\[
  \mathfrak{S}\left( \mathfrak{K}(\lambda) \right) : = \left\{ \sigma \in \mathfrak{S}_d \mid i \in K \in \mathfrak{K}(\lambda) \Longrightarrow \sigma(i) \in K \right\}.
\]
Here, note that we also have 
$\mathfrak{S}\left( \mathfrak{K}(\lambda) \right) 
= \left\{ \sigma \in \mathfrak{S}_d \mid \lambda_{\sigma(i)} = \lambda_i \text{ for every } 1 \leq i \leq d \right\}$.
Moreover, $\mathfrak{S}\left( \mathfrak{K}(\lambda) \right)$ is a subgroup of $\mathfrak{S}_d$ and is isomorphic to 
$\prod_{K \in \mathfrak{K}(\lambda)} \mathrm{Aut}(K) \cong \prod_{K \in \mathfrak{K}(\lambda)} \mathfrak{S}_{\#K}$.

The group $\mathfrak{S}\left( \mathfrak{K}(\lambda) \right)$ naturally acts on $\widetilde{\Sigma}_d(\lambda)$ 
by the permutation of coordinates, and its action is free.
Moreover, for $\zeta, \zeta' \in \widetilde{\Sigma}_d(\lambda)$, 
the equality $\widehat{\pi}(\lambda)(\zeta) = \widehat{\pi}(\lambda)(\zeta')$ holds if and only if 
the equality $\zeta' = \sigma \cdot \zeta$ holds for some $\sigma \in \mathfrak{S}\left( \mathfrak{K}(\lambda) \right)$, 
which can be verified by a similar argument to the proof of Lemma~4.5~(6) in~\cite{sugi1}.
We therefore have the bijection
\[
    \overline{\widehat{\pi}(\lambda)} : \widetilde{\Sigma}_d(\lambda) / \mathfrak{S}\left( \mathfrak{K}(\lambda) \right) 
    \cong \widehat{\Phi}_d^{-1}(\bar{\lambda}),
\]
which implies the equality
\begin{equation}\label{eq5.2}
  \#\left( \widehat{\Phi}_d^{-1}(\bar{\lambda}) \right) 
  = \frac{\#\left( \widetilde{\Sigma}_d(\lambda) \right)}{\#\left( \mathfrak{S}\left( \mathfrak{K}(\lambda) \right) \right)}
  = \frac{\#\left( \widetilde{\Sigma}_d(\lambda) \right)}{\prod_{K \in \mathfrak{K}(\lambda)}\left(\#K\right)!}.
\end{equation}

Combining equations~(\ref{eq5.1}) and~(\ref{eq5.2}), we have 
\[
   \#\left( \widehat{\Phi}_d^{-1}(\bar{\lambda}) \right) 
  = \frac{(d-1)s_d(\lambda)}{\prod_{K \in \mathfrak{K}(\lambda)}\left(\#K\right)!},
\]
which completes the proof of Theorem~\ref{mainthm2}.

\section*{Appendix}\label{appendix}

\setcounter{theorem}{0}
\setcounter{section}{1}
\renewcommand{\thesection}{\Alph{section}}
\setcounter{equation}{0}

In Appendix, we explain why we could find out the formula~(\ref{eq2.5}) in Theorem~\ref{mainthm1} 
by a careful look at equations~(\ref{eq2.3}) and~(\ref{eq2.4}).

First, 
we put
\[
	\tilde{e}_{\mathbb{I}} := \prod_{I \in \mathbb{I}} \bigl( \# I  -1 \bigr)!
\]
for each $\mathbb{I} \in \mathfrak{I}(\lambda)$.
Then by equation~(\ref{eq2.4}), we have $e_{\mathbb{I}}(\lambda) = \tilde{e}_{\mathbb{I}}$ for maximal $\mathbb{I} \in \mathfrak{I}(\lambda)$, 
and in general we have
\stepcounter{equation}
\begin{equation}\tag*{$(\mathrm \theequation)_{\mathbb{I}}$}\label{eqA.1}
	 e_{\mathbb{I}}(\lambda) =
	 \tilde{e}_{\mathbb{I}}
	  - \sum_\textrm{\scriptsize $\begin{matrix}
				       \mathbb{I}' \in \mathfrak{I}(\lambda) \\
				       \mathbb{I}' \succ \mathbb{I}, \;
				       \mathbb{I}' \ne \mathbb{I}
				     \end{matrix}$}
	 \left(
	  e_{\mathbb{I}'}(\lambda) \cdot \prod_{I \in \mathbb{I}}
	  \left( \prod_{k=\# I - \chi_I(\mathbb{I}')+1 }^{\# I -1}k \right)
	 \right)
\end{equation}
for every $\mathbb{I} \in \mathfrak{I}(\lambda)$.
Substituting equations~$(\mathrm \theequation)_{\mathbb{I}'}$ into equation~\ref{eqA.1} successively
for every pair $\mathbb{I}', \mathbb{I} \in \mathfrak{I}(\lambda)$ with $\mathbb{I}' \succ \mathbb{I}$,
we obtain the equalities in which every $e_{\mathbb{I}}(\lambda)$ is expressed in a linear combination of $\tilde{e}_{\mathbb{I}'}$ with $\mathbb{I}' \succ \mathbb{I}$.
Finally, substituting all of them into equation~(\ref{eq2.3}),
we obtain the equality in which $s_d(\lambda)$ is expressed in a linear combination of $(d-2)!$ and $\tilde{e}_{\mathbb{I}}$ for $\mathbb{I} \in \mathfrak{I}(\lambda)$,
which is expected to be equivalent to equation~(\ref{eq2.5}).

Let us execute the calculation explained above for some examples.
Note that equation~(\ref{eq2.5}) is equivalent to
\begin{equation}\label{eqA.2}
   s_d(\lambda) = (d-2)! - \sum_{\mathbb{I}\in \mathfrak{I}(\lambda)} 
				\left( \left\{ -(d-1) \right\}^{\#\mathbb{I} - 2} 
				\cdot \tilde{e}_{\mathbb{I}} \right).
\end{equation}

\begin{example}\label{exA.1}
 Let us consider the case in which $\mathbb{I}=\{ I_1, I_2, I_3 \}$ is the unique maximal element of $\mathfrak{I}(\lambda)$.
 In this case, we have $\mathfrak{I}(\lambda) = \left\{ \mathbb{I}, \mathbb{I}_1, \mathbb{I}_2, \mathbb{I}_3 \right\}$, where
 \[
	\mathbb{I}_1 = \{ I_1, I_2 \amalg I_3 \}, \quad
	\mathbb{I}_2 = \{ I_2, I_1 \amalg I_3 \} \quad \textrm{and} \quad
	\mathbb{I}_3 = \{ I_3, I_1 \amalg I_2 \}.
 \]
 We put $\#I_u = i_u$ for $1 \leq u \leq 3$.  Note that we have $i_1 + i_2 + i_3 = d$.
 By equation~\ref{eqA.1}, we have
 $e_{\mathbb{I}}(\lambda) = \tilde{e}_{\mathbb{I}}$ and 
 $e_{\mathbb{I}_1}(\lambda) = \tilde{e}_{\mathbb{I}_1} - e_{\mathbb{I}}(\lambda) \cdot \left( i_2+i_3-1 \right) = \tilde{e}_{\mathbb{I}_1} - \left( d - i_1 - 1 \right) \cdot \tilde{e}_{\mathbb{I}}$.

 Hence, by equation~(\ref{eq2.3}), we have
 \begin{align*}
	s_d(\lambda) &= (d-2)! - \left\{ e_{\mathbb{I}}(\lambda)\cdot(d-2) + 
	e_{\mathbb{I}_1}(\lambda) + e_{\mathbb{I}_2}(\lambda) + e_{\mathbb{I}_3}(\lambda) \right\}\\
	&= (d-2)! - \left[ (d-2) \cdot \tilde{e}_{\mathbb{I}} + 
		\sum_{u=1}^3 \left\{ 
			\tilde{e}_{\mathbb{I}_u} - \left( d-i_u-1 \right) \cdot \tilde{e}_{\mathbb{I}}
		 \right\} \right]\\
	&= (d-2)! - \tilde{e}_{\mathbb{I}_1} - \tilde{e}_{\mathbb{I}_2} - \tilde{e}_{\mathbb{I}_3} + (d-1) \cdot \tilde{e}_{\mathbb{I}},
 \end{align*}
 which is the same as equation~(\ref{eqA.2}) in this case.
\end{example}

\begin{example}\label{exA.2}
 Let us also consider the case in which $\mathbb{I}=\{ I_1, I_2, I_3, I_4 \}$ is the unique maximal element of $\mathfrak{I}(\lambda)$.
 In this case, $\mathfrak{I}(\lambda)$ has $14$ elements, and a similar calculation to Example~\ref{exA.1} is much more complicated than the calculation in Example~\ref{exA.1}.
 We only consider the coefficient of $\tilde{e}_{\mathbb{I}}$.

 We put $\#I_u = i_u$ for $1 \leq u \leq 4$. Note that we have $i_1 + i_2 + i_3 + i_4 = d$.
 We can express
 \[
     \mathfrak{I}(\lambda) = \left\{ \mathbb{I}, \mathbb{I}_{\{1,2\}}, \mathbb{I}_{\{3,4\}}, \mathbb{I}_{\{1,3\}}, \mathbb{I}_{\{2,4\}}, \mathbb{I}_{\{1,4\}}, \mathbb{I}_{\{2,3\}}, \mathbb{I}_1, \mathbb{I}_2, \mathbb{I}_3, \mathbb{I}_4, \mathbb{I}_5, \mathbb{I}_6, \mathbb{I}_7 \right\},
 \]
 where $\mathbb{I}_{\{1,2\}} = \left\{ I_1 \amalg I_2, I_3, I_4 \right\},\ 
 \mathbb{I}_1 = \left\{ I_1, I_2 \amalg I_3 \amalg I_4 \right\}$ and
 $\mathbb{I}_5 = \{ I_1 \amalg I_2, I_3 \amalg I_4 \}$ for instance.
 Here, $\mathbb{I}_1,\dots, \mathbb{I}_7$ denote the same as in Example~\ref{ex4.4}.

 By equation~\ref{eqA.1}, we have $e_{\mathbb{I}}(\lambda) = \tilde{e}_{\mathbb{I}}$ and
 $e_{\mathbb{I}_{\{1,2\}}}(\lambda) = \tilde{e}_{\mathbb{I}_{\{1,2\}}} - \left( i_1 + i_2 - 1 \right) \cdot \tilde{e}_{\mathbb{I}}$.
 For $\mathbb{I}_1$, we have
 $\left\{ \mathbb{I}' \in \mathfrak{I}(\lambda) \mid \mathbb{I}' \succ \mathbb{I}_1 \right\}
 = \left\{ \mathbb{I}, \mathbb{I}_{\{2,3\}}, \mathbb{I}_{\{3,4\}}, \mathbb{I}_{\{2,4\}}, \mathbb{I}_1 \right\}$, which implies
 \begin{align*}
    e_{\mathbb{I}_1}(\lambda) &= \tilde{e}_{\mathbb{I}_1} 
	- \Bigl[  
		e_{\mathbb{I}}(\lambda) \cdot \left( i_2+i_3+i_4-1 \right)\left( i_2+i_3+i_4-2 \right)\\
		&\hspace{45pt} + \left\{  
			e_{\mathbb{I}_{\{2,3\}}}(\lambda) + e_{\mathbb{I}_{\{3,4\}}}(\lambda) + e_{\mathbb{I}_{\{2,4\}}}(\lambda)
		\right\} \cdot \left( i_2+i_3+i_4-1 \right)
	\Bigr]\\
	&= \tilde{e}_{\mathbb{I}_1} 
	- \Bigl[  
		\left( i_2+i_3+i_4-1 \right)\left( i_2+i_3+i_4-2 \right) 
			\cdot \tilde{e}_{\mathbb{I}}+ \left( i_2+i_3+i_4-1 \right)\\
		&\hspace{10pt}  \times
		\left\{  
			\tilde{e}_{\mathbb{I}_{\{2,3\}}} 
				- \left( i_2 + i_3 - 1 \right) \cdot \tilde{e}_{\mathbb{I}}
			+ \tilde{e}_{\mathbb{I}_{\{3,4\}}} 
				- \left( i_3 + i_4 - 1 \right) \cdot \tilde{e}_{\mathbb{I}}
			+ \tilde{e}_{\mathbb{I}_{\{2,4\}}} 
				- \left( i_2 + i_4 - 1 \right) \cdot \tilde{e}_{\mathbb{I}}
		\right\}
	\Bigr]\\
	&= \tilde{e}_{\mathbb{I}_1} - \left( i_2+i_3+i_4-1 \right) \left\{
			\tilde{e}_{\mathbb{I}_{\{2,3\}}} + \tilde{e}_{\mathbb{I}_{\{3,4\}}} +
			 \tilde{e}_{\mathbb{I}_{\{2,4\}}} 
		\right\} + \left( i_2+i_3+i_4-1 \right)\\
	&\hspace{20pt} \times \left\{
			-\left( i_2+i_3+i_4-2 \right) +\left( i_2 + i_3 - 1 \right) 
			+ \left( i_3 + i_4 - 1 \right) + \left( i_2 + i_4 - 1 \right)
		\right\} \cdot \tilde{e}_{\mathbb{I}}\\
	&= \tilde{e}_{\mathbb{I}_1} - \left( i_2+i_3+i_4-1 \right) \left\{
			\tilde{e}_{\mathbb{I}_{\{2,3\}}} + \tilde{e}_{\mathbb{I}_{\{3,4\}}} +
			 \tilde{e}_{\mathbb{I}_{\{2,4\}}} 
		\right\} + \left( i_2+i_3+i_4-1 \right)^2 \cdot \tilde{e}_{\mathbb{I}}.
 \end{align*}
 For $\mathbb{I}_5$, we have
 $\left\{ \mathbb{I}' \in \mathfrak{I}(\lambda) \mid \mathbb{I}' \succ \mathbb{I}_5 \right\}
 = \left\{ \mathbb{I}, \mathbb{I}_{\{1,2\}}, \mathbb{I}_{\{3,4\}}, \mathbb{I}_5 \right\}$, which implies
 \begin{align*}
  e_{\mathbb{I}_5}(\lambda) &= \tilde{e}_{\mathbb{I}_5}
	- \Bigl[
		e_{\mathbb{I}}(\lambda) \cdot \left( i_1+i_2-1 \right)\left( i_3+i_4-1 \right)\\
	&\hspace{70pt} + e_{\mathbb{I}_{\{1,2\}}}(\lambda) \cdot \left( i_3+i_4-1 \right)
		+ e_{\mathbb{I}_{\{3,4\}}}(\lambda) \cdot \left( i_1+i_2-1 \right)
	\Bigr]\\
	&= \tilde{e}_{\mathbb{I}_5}
	- \Bigl[
		\left( i_1+i_2-1 \right)\left( i_3+i_4-1 \right) \cdot \tilde{e}_{\mathbb{I}}
		+ \left( i_3+i_4-1 \right) \cdot \left\{
			\tilde{e}_{\mathbb{I}_{\{1,2\}}} - \left( i_1 + i_2 - 1 \right) \cdot \tilde{e}_{\mathbb{I}}
		\right\}\\
		&\hspace{70pt}+ \left( i_1+i_2-1 \right) \cdot \left\{
			\tilde{e}_{\mathbb{I}_{\{3,4\}}} - \left( i_3 + i_4 - 1 \right) \cdot \tilde{e}_{\mathbb{I}}
		\right\}
	\Bigr]\\
	&= \tilde{e}_{\mathbb{I}_5} - \left( i_3+i_4-1 \right) \cdot \tilde{e}_{\mathbb{I}_{\{1,2\}}}
						- \left( i_1+i_2-1 \right) \cdot \tilde{e}_{\mathbb{I}_{\{3,4\}}}
		+ \left( i_1+i_2-1 \right)\left( i_3+i_4-1 \right) \cdot \tilde{e}_{\mathbb{I}}.
 \end{align*}
 By equation~(\ref{eq2.3}), we have
 \begin{align*}
	s_d(\lambda) &= (d-2)! - (d-2)(d-3) \cdot e_{\mathbb{I}}(\lambda)\\
	 &\hspace{15pt}- (d-2)\left\{ e_{\mathbb{I}_{\{1,2\}}}(\lambda) + e_{\mathbb{I}_{\{3,4\}}}(\lambda) + e_{\mathbb{I}_{\{1,3\}}}(\lambda) + e_{\mathbb{I}_{\{2,4\}}}(\lambda) + e_{\mathbb{I}_{\{1,4\}}}(\lambda) + e_{\mathbb{I}_{\{2,3\}}}(\lambda) \right\}\\
	&\hspace{50pt}- \left\{ e_{\mathbb{I}_{1}}(\lambda) + e_{\mathbb{I}_{2}}(\lambda) + e_{\mathbb{I}_{3}}(\lambda) + e_{\mathbb{I}_{4}}(\lambda) + e_{\mathbb{I}_{5}}(\lambda) + e_{\mathbb{I}_{6}}(\lambda) + e_{\mathbb{I}_{7}}(\lambda) \right\}.
 \end{align*}
 Therefore, when $s_d(\lambda)$ is expressed as a linear combination of $(d-2)!$ and $\tilde{e}_{\mathbb{I}'}$ for $\mathbb{I}' \in \mathfrak{I}(\lambda)$,
 the coefficient of $\tilde{e}_{\mathbb{I}}$ equals
 \begin{align*}
   &-(d-2)(d-3) + (d-2)\bigl\{ \left( i_1+i_2-1 \right) + \left( i_3+i_4-1 \right) + \left( i_1+i_3-1 \right)\\
   &\hspace{135pt}+ \left( i_2+i_4-1 \right) + \left( i_1+i_4-1 \right) + \left( i_2+i_3-1 \right) \bigr\}\\
   &-\left( i_2+i_3+i_4-1 \right)^2 -\left( i_1+i_3+i_4-1 \right)^2 -\left( i_1+i_2+i_4-1 \right)^2 -\left( i_1+i_2+i_3-1 \right)^2\\
   &- \left( i_1+i_2-1 \right)\left( i_3+i_4-1 \right) - \left( i_1+i_3-1 \right)\left( i_2+i_4-1 \right) - \left( i_1+i_4-1 \right)\left( i_2+i_3-1 \right)\\
	&= -(d-2)(d-3) + (d-2)(3d-6) -\left\{ \sum_{u=1}^4\left( d-i_u-1 \right)^2
		+ 2\sum_{u<v}i_ui_v -3d + 3 \right\}\\
	&= 2d^2 - 7d + 6 - \left\{ 4(d-1)^2 -2(d-1)(i_1+i_2+i_3+i_4) + (i_1+i_2+i_3+i_4)^2 -3d+3 \right\}\\
	&= -(d-1)^2,
 \end{align*}
 which assures equation~(\ref{eqA.2}) in this case.
\end{example}

By a similar calculation to Examples~\ref{exA.1} and~\ref{exA.2} for some other examples,
we can state Theorem~\ref{mainthm1}.
However, we could not prove Theorem~\ref{mainthm1} by a similar calculation to Examples~\ref{exA.1} and~\ref{exA.2}.
Hence, we took another method, which was already described precisely 
in Section~\ref{proof1}.

\end{document}